\documentclass[11pt,a4paper]{article} %article文档

\usepackage[ruled,linesnumbered]{algorithm2e} %排版算法
\usepackage{amscd} %简单交换图
\usepackage{amsfonts} %字体
\usepackage{amsmath} %数学公式
\usepackage{amssymb} %更多公式符号
\usepackage{amsthm} %定理环境
\usepackage{bm} %允许公式粗体
\usepackage{booktabs} %改善表格
\usepackage{delarray} %矩阵
\usepackage{extarrows} %箭头
\usepackage{geometry} %页面布局
\usepackage{graphicx} %插图
\usepackage{hyperref} %超链接
\usepackage{longtable} %长表格
\usepackage{tikz-cd} %交换图
\usepackage[all]{xy} %交换图
\usepackage[affil-it]{authblk} %标题

\usepackage{adjustbox}
\usepackage{booktabs} %美化表格（可选）
\usepackage{caption}
\usepackage{color}
\usepackage{diagbox} %用来画斜线的宏包
\usepackage{enumitem} %编号
\usepackage{float}
\usepackage{hyperref}%创建超文本链接和PDF书签
\usepackage{indentfirst} %首行缩进
\usepackage{makecell} %表格单元格换行
\usepackage{mathrsfs} %数学字体宏包
\usepackage{mathtools} %数学工具宏包
\usepackage{multirow} %跨行表格宏包
\usepackage{relsize} %缩放宏包
\usepackage{subfigure}

\setlist[enumerate]{itemsep=0pt,parsep=0pt}

\title{On the $K$-extensions between Serre weights for unramified $\GL_3$}
\author{Yitong Wang\thanks{E-mail address: \texttt{yitongw.wang@utoronto.ca}}}
\date{}
\allowdisplaybreaks[4] %多行公式换页

\def\FF{\mathbb{F}}
\def\GG{\mathbb{G}}

\def\ZZ{\mathbb{Z}}

\def\NNN{\mathbb{Z}_{\geq0}}

\def\cJ{\mathcal{J}}

\def\dim{\operatorname{dim}}

\def\Ext{\operatorname{Ext}} %%% new

\def\GL{\operatorname{GL}}

\def\Hom{\operatorname{Hom}}
\def\id{\operatorname{id}}

\def\Ind{\operatorname{Ind}}

\def\JH{\operatorname{JH}}
\def\Ker{\operatorname{Ker}}

\def\Mat{\operatorname{Mat}}
\def\max{\operatorname{max}}

\def\min{\operatorname{min}}
\def\mod{\operatorname{mod}}

\def\SL{\operatorname{SL}}

\def\SL{\operatorname{SL}}
\def\soc{\operatorname{soc}}

\def\eqdef{\overset{\mathrm{def}}{=}}
\def\Fp{\mathbb{F}_p}

\def\Fq{\mathbb{F}_q}
\def\iff{\Leftrightarrow}

\def\into{\hookrightarrow}
\def\inv{^{-1}}
\def\ism{\stackrel{\sim}{\rightarrow}}

\def\OL{\mathcal{O}_L}

\def\Qp{\mathbb{Q}_p}

\def\sprod{\scalebox{1}{$\prod$}}
\def\ssum{\scalebox{1}{$\sum$}}

\def\x{^{\times}}

\newcommand{\ang}[1]{\langle#1\rangle}

\newcommand{\bbbra}[1]{\left[#1\right]}
\newcommand{\bbra}[1]{\left(#1\right)}

\newcommand{\bigbbbra}[1]{\big[#1\big]}
\newcommand{\bigbra}[1]{\big(#1\big)}

\newcommand{\bigset}[1]{\big\{#1\big\}}
\newcommand{\bra}[1]{(#1)}

\newcommand{\ddbra}[1]{[\![#1]\!]}

\newcommand{\op}[1]{\operatorname{#1}}
\newcommand{\ovl}[1]{\overline{#1}}

\newcommand{\set}[1]{\{ #1 \}}
\newcommand{\smat}[1]{\left(\begin{smallmatrix}#1\end{smallmatrix}\right)}
\newcommand{\sset}[1]{\left\{ #1 \right\}}
\newcommand{\un}[1]{\underline{#1}}
\newcommand{\wh}[1]{\widehat{#1}}

\begin{document}

\newtheorem{definition}{Definition}[section] %%编号以section为单位
\newtheorem{remark}[definition]{Remark}
\newtheorem{example}[definition]{Example}
\newtheorem{proposition}[definition]{Proposition}
\newtheorem{lemma}[definition]{Lemma}
\newtheorem{corollary}[definition]{Corollary}
\newtheorem{theorem}[definition]{Theorem}
\newtheorem{conjecture}[definition]{Conjecture}

\maketitle

\begin{abstract}
    Let $p\geq5$ be a prime number. Let $L$ be a finite unramified extension of $\Qp$ with ring of integers $\OL$ and residue field $\Fq$. Given two Serre weights for $\GL_3(\Fq)$, we prove that in most cases the extensions between them for $\GL_3(\OL)$ modulo the center coincide with their $\GL_3(\Fq)$-extensions.
\end{abstract}

\tableofcontents

\section{Introduction}

Let $p$ be a prime number. Let $L$ be a finite unramified extension of $\Qp$ with ring of integers $\OL$ and residue field $\Fq$. The mod $p$ Langlands correspondence is completely understood for $\GL_2(\Qp)$, but remains largely mysterious for $\GL_2(L)$ when $L\neq\Qp$, as well as for $\GL_3(L)$.

By contrast, mod $p$ representations of the compact group $\GL_2(\OL)$ are more tractable. We know that irreducible smooth mod $p$ representations of $\GL_2(\OL)$ necessarily factor through $\GL_2(\Fq)$ and are known as Serre weights for $\GL_2(\Fq)$. A natural question is therefore to understand extensions between two Serre weights. 

Firstly, the $\GL_2(\Fq)$-extensions between two Serre weights are completely studied in \cite[Cor.~5.6(i)]{BP12}. Secondly, it is known that in most cases, the extensions for $\GL_2(\OL)$ modulo the center between two Serre weights coincide with their $\GL_2(\Fq)$-extensions, which is proved in different cases by \cite[Prop.~2.21]{Hu10}, \cite[Lemma~2.10]{HW22} and \cite[Lemma~3.22]{HW24}. These two  results together provide a framework for the explicit description of many mod $p$ representations of $\GL_2(\OL)$, which play an important role in the recent progress of the mod $p$ Langlands correspondence for $\GL_2(L)$ (see \cite{BHHMS1} and \cite{HW22}).

In this paper, we study the analogous question for $\GL_3$. More precisely, we investigate the relationship between $\GL_3(\OL)$-extensions and $\GL_3(\Fq)$-extensions between two Serre weights for $\GL_3(\Fq)$. We let $Z_1\cong1+p\OL$ be the center of $\GL_3(\OL)$. Our main result is the following.

\begin{theorem}\label{K-ext Thm intro}
    Suppose that $p\geq5$. Let $\sigma$ and $\tau$ be two Serre weights for $\GL_3(\Fq)$. Apart from some exceptional cases (see Theorem \ref{K-ext Thm main} for a precise statement), we have an isomorphism
    \begin{equation*}
        \Ext^1_{\GL_3(\OL)/Z_1}(\tau,\sigma)\cong\Ext^1_{\GL_3(\Fq)}(\tau,\sigma).
    \end{equation*}
\end{theorem}

We expect that the exceptional cases in Theorem \ref{K-ext Thm intro} are essential, although we do not investigate this point in the present paper. 

\vspace{1em}

The proof of Theorem \ref{K-ext Thm intro} is technical and requires new ingredients compared with the $\GL_2$-case. To give a brief explanation, we denote by $H$ the set of diagonal matrices with entries in $\Fq\x$. %, where $[\bullet]$ denotes the Teichm\"uller lift. 
Each Serre weight can be decomposed as a direct sum of $H$-eigenspaces. In the $\GL_2$-case, we show that certain elements in a Serre weight must be zero by proving that they have $H$-eigencharacters not appearing in the Serre weight. However, in the $\GL_3$-case, for the upper alcove Serre weights, they contain too many $H$-eigencharacters so that a similar argument fails. To overcome this problem, we regard a Serre weight as an algebraic representation of $\GL_3$ and study the weight spaces instead of the $H$-eigenspaces, which give better control of where certain elements belong to. We refer to \S\S\ref{K-ext sec argument}-\ref{K-ext sec conclusion} for a more detailed outline of the argument.

Finally, we remark that in the sufficiently generic case in the sense of \cite{LLHLM20}, the $\GL_3(\Fq)$-extensions between two Serre weights are studied in \cite[Lemma~4.2.6]{LLHLM20}. Together with Theorem \ref{K-ext Thm intro}, this provides a fairly complete understanding of the $\GL_3(\OL)$-extensions between two Serre weights for $\GL_3(\Fq)$. We expect that Theorem \ref{K-ext Thm intro} will be useful in the mod $p$ Langlands correspondence for $\GL_3(L)$ in the future.

\subsection*{Organization of the paper}

In \S\ref{K-ext Sec Tensor}, we describe the structure of the tensor products of two Serre weights. In \S\S\ref{K-ext sec operator Fp}-\ref{K-ext sec operator Fq}, we study certain operators acting on Serre weights and obtain some technical lemmas involving the weight spaces. In \S\S\ref{K-ext sec argument}-\ref{K-ext sec conclusion}, we combine the results of the previous sections and complete the proof of Theorem \ref{K-ext Thm intro}.

\subsection*{Acknowledgements}

We thank Florian Herzig, Yongquan Hu and Zicheng Qian for helpful discussions. 

The author was supported by the University of Toronto. We thank the Morningside Center of Mathematics for their hospitality and providing excellent working conditions where part
of this work was carried out.

\subsection*{Notation}

Let $p\geq5$ be a prime number. Let $L$ be an unramified extension of $\Qp$ of degree $f$ with ring of integers $\OL$ and residue field $\Fq$ (hence $q=p^f$). 

Let $\FF$ be a large enough finite extension of $\Fp$. Fix an embedding $\sigma_0:\Fq\into\FF$ and let $\sigma_j\eqdef\sigma_0\circ\varphi^j$ for $j\in\ZZ$, where $\varphi:x\mapsto x^p$ is the arithmetic Frobenius on $\Fq$. We identify $\cJ\eqdef\Hom(\Fq,\FF)$ with $\set{0,1,\ldots,f-1}$, which is also identified with $\ZZ/f\ZZ$ so that the addition and subtraction in $\cJ$ are modulo $f$.

For $a\in\OL$, we write $\ovl{a}\in\Fq$ for the reduction modulo $p$, and we also view it as an element of $\FF$ via $\sigma_0$. For $b\in\Fq$, we write $[b]\in\OL=W(\Fq)$ for its Teichm\"uller lift. For $a,b,c\in\OL$, we write $\op{diag}(a,b,c)$ for the diagonal matrix $\smat{a&&\\&b&\\&&c}$.

All representations are defined over $\FF$, unless stated otherwise. If $V$ a finite-dimensional representation of a group, we let $\JH(V)$ denote the set of Jordan--H\"older factors of $V$.

For $P$ a statement, we let $\delta_P\eqdef1$ if $P$ is true and $\delta_{P}\eqdef0$ otherwise. 

\section{Tensor products of Serre weights}\label{K-ext Sec Tensor}

Let $G\eqdef\GL_{3/\Fp}$ and $T\subseteq G$ be the maximal split torus consisting of diagonal matrices. Let $X(T)$ denote the group of characters of $T$, which we identify with $\ZZ^3$ in the usual way. We often say a weight for an element of $X(T)$. For $\lambda\in X(T)$, we write $\lambda=(\lambda_1,\lambda_2,\lambda_3)\in\ZZ^3$. We write $\Phi$ (resp.\;$\Phi^{\vee}$) for the set of roots (resp.\;coroots) of $(G,T)$. %For simplicity, we write $G(\Fp)\eqdef\GL_3(\Fp)$ and so on.

Let $B\subseteq G$ be the Borel subgroup consisting of upper-triangular matrices and $U\subseteq B$ be the maximal unipotent subgroup. Let $\Phi^+\subseteq\Phi$ be the set of positive roots with respect to $(B,T)$. Concretely, for $1\leq i\neq j\leq 3$ we write $\alpha_{ij}\eqdef\varepsilon_i-\varepsilon_j\in X(T)$ with $\varepsilon_1\eqdef(1,0,0)$, $\varepsilon_2\eqdef(0,1,0)$, and $\varepsilon_3\eqdef(0,0,1)$. Then we have $\Phi^+=\set{\alpha_{12},\alpha_{23},\alpha_{13}}$. The set $\Phi^+$ induces a partial order on $X(T)$: for $\lambda,\mu\in X(T)$ we say that $\lambda\leq\mu$ if $\mu-\lambda\in\sum_{\alpha\in\Phi^+}\NNN\alpha$. We also write $\Phi^-\eqdef-\Phi^+$ so that $\Phi=\Phi^+\sqcup\Phi^-$. 

We let $X(T)_+\eqdef\set{\lambda\in X(T):\lambda_1\geq\lambda_2\geq\lambda_3}$ be the set of dominant weights. For $\lambda\in X(T)_+$, we have the following finite-dimensional $G$-modules (\cite{Jan03}):
\begin{equation*}
\begin{aligned}
    &L(\lambda):~\text{simple module of highest weight $\lambda$};\\
    &V(\lambda):~\text{Weyl module of highest weight $\lambda$};\\
    &T(\lambda):~\text{indecomposable tilting module of highest weight $\lambda$}.
\end{aligned}
\end{equation*}
For $\lambda\in X(T)\setminus X(T)_+$, we set $L(\lambda)=V(\lambda)=T(\lambda)\eqdef0$. For $V$ a finite-dimensional $G$-module, we denote by $[V]$ the corresponding element in the Grothendieck group of $G$-modules.

For $n\geq0$, we define the set $X_n(T)$ of $p^n$-restricted weights %and the set $X_0(T)$ of central weights as
\begin{equation*}
\begin{aligned}
    X_n(T)&\eqdef\set{\lambda\in X(T):0\leq\lambda_1-\lambda_2,\lambda_2-\lambda_3\leq p^n-1}.%;\\
    %X_0(T)&\eqdef\set{\lambda\in X(T):\lambda_1=\lambda_2=\lambda_3}.
\end{aligned}
\end{equation*}
%We let $X_1(T)\eqdef\set{\lambda\in X(T):0\leq\lambda_1-\lambda_2,\lambda_2-\lambda_3\leq p-1}$ be the set of $p$-restricted weights and let $X_0(T)\eqdef\set{\lambda\in X(T):\lambda_1=\lambda_2=\lambda_3}$ be the set of central weights. %For $\lambda\in X_1(T)$, we write $F(\lambda)\eqdef L(\lambda)|_{G(\Fp)}$ the corresponding Serre weight of $G(\Fp)$. 
For $\lambda\in X_f(T)$, we write $\lambda=\sum_{j=0}^{f-1}p^j\lambda_j$ with $\lambda_j\in X_1(T)$, with an abuse of notation. By \cite[II.3.16]{Jan03} we have $L(\lambda)\cong\otimes_{j=0}^{f-1}L(\lambda_j)^{[j]}$, where $[j]$ denotes the $j$-th Frobenius twist. By \cite[Cor.~3.17]{Her09}, when $\lambda$ runs through $X_f(T)$, the restrictions $F(\lambda)\eqdef L(\lambda)|_{G(\Fq)}$ exhaust all the irreducible representations of $G(\Fq)$. Moreover, for two weights $\lambda,\lambda'\in X_f(T)$, we have $F(\lambda)\cong F(\lambda')$ if and only if $\lambda-\lambda'\in(q-1)X_0(T)$.

For $\lambda\in X(T)$, we write $\lambda|_{\SL_3}\eqdef(\lambda_1-\lambda_2,\lambda_2-\lambda_3)\in\ZZ^2$. 

\begin{definition}
    We define the following disjoint subsets of $X(T)_{+}$:
    \begin{equation*}
    \begin{aligned}
        C(1)&\eqdef\set{\lambda\in X_1(T):\lambda_1-\lambda_3\leq p-3};\\
        C(1|2)&\eqdef\set{\lambda\in X_1(T):\lambda_1-\lambda_3=p-2};\\
        C(2)&\eqdef\set{\lambda\in X_1(T):0\leq\lambda_1-\lambda_2,\lambda_2-\lambda_3\leq p-2,~\lambda_1-\lambda_3\geq p-1};\\
        C(2|3)&\eqdef\set{\lambda\in X_1(T):\lambda_1-\lambda_2=p-1~\text{or}~\lambda_2-\lambda_3=p-1}.
    \end{aligned}
    \end{equation*}
    \begin{equation*}
    \begin{aligned}
        C(3)&\eqdef\set{\lambda\in X(T)_{+}:\lambda_1-\lambda_2\in\set{p,p+1},\lambda_1-\lambda_3\leq2p-2};\\
        C(3')&\eqdef\set{\lambda\in X(T)_{+}:\lambda_2-\lambda_3\in\set{p,p+1},\lambda_1-\lambda_3\leq2p-2};\\
        C(p)&\eqdef\set{\lambda\in X(T)_{+}:\lambda|_{\SL_3}\in\set{(p,p-1),(p-1,p),(p,p)}}.
    \end{aligned}
    \end{equation*}
    In particular, we have $X_1(T)=C(1)\sqcup C(1|2)\sqcup C(2)\sqcup C(2|3)$. 
\end{definition}

%\red{For convenience, we work in the representation ring $\cR=\Rep_{\FF}(\GL_3)$, which is the quotient of the free abelian group on the set $[L(\lambda)]$, as $\lambda$ varies over $X(T)_+$, by the subgroup generated by all expressions of the form $[M]-[M']-[M'']$ such that $0\to M'\to M\to M''\to0$ is a short exact sequence of finite-dimensional $G$-modules.}

\begin{definition}
    We define the following reflections:
\begin{equation}\label{K-ext Eq reflections}
\begin{aligned}
    s_{2}:X(T)\to X(T),\quad(a,b,c)\mapsto(c+p-2,b,a-p+2);\\
    s_{3}:X(T)\to X(T),\quad(a,b,c)\mapsto(b+p-1,a-p+1,c);\\
    s_{3'}:X(T)\to X(T),\quad(a,b,c)\mapsto(a,c+p-1,b-p+1).\\
\end{aligned}
\end{equation}
Moreover, for $\bullet\in\set{2,3,3'}$ and $\lambda\in C(\bullet)$, we write $s(\lambda)\eqdef s_{\bullet}(\lambda)\in X_1(T)$.
\end{definition}

By \cite[Prop.~3.18]{Her09} we have $V(\lambda)=L(\lambda)$ if $\lambda\in C(1)\cup C(1|2)\cup C(2|3)$ and $[L(\lambda)]=[V(\lambda)]-[V(s_2(\lambda))]$ in the Grothendieck group of $G$-modules if $\lambda\in C(2)$.

\begin{proposition}\label{K-ext Prop Tilting Structure}
    %For $\bullet\in\set{2,3,3'}$ and $\mu\in C(\bullet)$, we write $s(\mu)\eqdef s_{\bullet}(\mu)$. 
    Let $\lambda\in X_1(T)\cup C(3)\cup C(3')\cup C(p)$. 
\begin{enumerate}
    \item 
    We have $\soc_{G}T(\lambda)\cong L(\lambda')$ with
    \begin{equation}\label{K-ext Eq lambda'}
        \lambda'=
    \begin{cases}
        \lambda&\text{if}~\lambda\in C(1)\cup C(1|2)\cup C(2|3)\\
        s(\lambda)&\text{if}~\lambda\in C(2)\cup C(3)\cup C(3')\\
        \lambda-\alpha_{13}&\text{if}~\lambda|_{\SL_3}=(p-1,p)~\text{or}~(p,p-1)\\
        \lambda-2\alpha_{13}&\text{if}~\lambda|_{\SL_3}=(p,p).
    \end{cases}
    \end{equation}
    In particular, $\lambda'\in X_1(T)$ in all the above cases. 
    \item 
    In the Grothendieck group of $G$-modules we have
    \begin{equation*}
        [T(\lambda)]=
    \begin{cases}
        [V(\lambda)]&\text{if}~\lambda\in C(1)\cup C(1|2)\cup C(2|3)\\
        [V(\lambda)]+[V\bra{s(\lambda)}]&\text{if}~\lambda\in C(2)\cup C(3)\cup C(3')\\
        [V(\lambda)]+[V(\lambda-\alpha_{12})]+[V(\lambda-\alpha_{13})]&\text{if}~\lambda|_{\SL_3}=(p,p-1)\\
        [V(\lambda)]+[V(\lambda-\alpha_{23})]+[V(\lambda-\alpha_{13})]&\text{if}~\lambda|_{\SL_3}=(p-1,p)\\
        \ssum_{\alpha\in\Phi}[V(\lambda-\alpha_{13}+\alpha)]&\text{if}~\lambda|_{\SL_3}=(p,p-1).
    \end{cases}
    \end{equation*}
\end{enumerate}
\end{proposition}

\begin{proof}
    This proposition is taken from \cite[Theorem~B]{BDM15}. Note in that paper they work with $\SL_3$. Using \cite[(4.3)]{LLHLM20} together with the isomorphisms $T(\lambda)|_{\SL_3}\cong T(\lambda|_{\SL_3})$ (see \cite[Lemma~1.4(i)]{Don93}), $V(\lambda)|_{\SL_3}\cong V(\lambda|_{\SL_3})$ and $L(\lambda)|_{\SL_3}\cong L(\lambda|_{\SL_3})$, we have completely similar results for $\GL_3$.
\end{proof}

\begin{proposition}\label{K-ext Prop Tensor Socle}
    Let $j_0\in\cJ$. Let $\lambda_{j_0}\in X_1(T)\cup C(3)\cup C(3')\cup C(p)$ and $\lambda'_{j_0}\in X_1(T)$ be as in \eqref{K-ext Eq lambda'}. Let $\lambda_j\in X_1(T)$ for $j\neq j_0$. We write $\lambda=\sum_{j=0}^{f-1}p^j\lambda_j\in X(T)$ and $\lambda'=p^{j_0}\lambda'_{j_0}+\sum_{j\neq j_0}p^j\lambda_j\in X_f(T)$. We also write $T^{j_0}(\lambda)\eqdef T(\lambda_{j_0})^{[j_0]}\otimes\!\bigbra{\!\otimes_{j\neq j_0}L(\lambda_{j})^{[j]}}$. Then we have $\soc_{G(\Fq)}\!\bigbra{T^{j_0}(\lambda)|_{G(\Fq)}}\cong F(\lambda')$.
    %\begin{equation*}
        %\soc_{G(\Fq)}\bbra{T(\lambda_{j_0})^{[j_0]}\otimes\!\bigbra{\!\otimes_{j\neq j_0}L(\lambda_{j})^{[j]}}}\cong F(\lambda').
    %\end{equation*}
\end{proposition}

\begin{proof}
    For $n\geq1$, we let $G_n$ denote the $n$-th Frobenius kernel of $G$. For $\mu\in X_n(T)$, we write $\wh{Q}_n(\mu)$ for the injective envelope of the irreducible representation $L(\mu)|_{G_nT}$ in the category of $G_nT$-modules. By \cite[II.11.11]{Jan03} and since $p\geq5$, $\wh{Q}_n(\mu)$ has a unique $G$-structure and we denote by $Q_n(\mu)$ in what follows.
    
    %First we prove that there is an injective $G$-equivariant embedding $T^{j_0}(\lambda)\into Q_f(\lambda')$. To do this, it suffices to show that $\Ext^1_{G}\bra{L(\mu),Q_f(\lambda')}=0$ for all $L(\mu)\in\JH(T^{j_0}(\lambda))$. We write $\mu=\mu_0+p^f\mu_1$ with $\mu_0\in X_f(T)$ and $\mu_1\in X(T)$. Then by \cite[II.3.16]{Jan03} we have $L(\mu)=L(\mu_0)\otimes L(\mu_1)^{[f]}$. Since $Q_f(\lambda')|_{G_f}$ is an injective $G_f$-module with $G_f$-socle isomorphic to $L(\lambda')|_{G_f}$, the Lyndon--Hochschild--Serre spectral sequence (\cite[I.6.6(1),~I.6.5(2)]{Jan03}) shows that
    %\begin{equation*}
    %\begin{aligned}
        %\Ext^1_{G}\bra{L(\mu),Q_f(\lambda')}&\cong\Ext^1_{G/G_f}\bigbra{L(\mu_1)^{[f]},\Hom_{G_f}\bra{L(\mu_0),Q_f(\lambda')}}\\
        %&\cong\Ext^1_{G/G_f}\bigbra{L(\mu_1)^{[f]},\Hom_{G_f}\bra{L(\mu_0),L(\lambda')}}\\
        %&\cong
    %\begin{cases}
        %\Ext^1_{G/G_f}\bigbra{L(\mu_1)^{[f]},L\bra{\lambda'-\mu_0}}&\text{if}~\lambda'-\mu_0\in p^fX_0(T)\\
        %0&\text{otherwise}
    %\end{cases}\\
        %&\cong
    %\begin{cases}
        %\Ext^1_{G}\bigbra{L(\mu_1),L\bra{(\lambda'-\mu_0)/p^f}}&\text{if}~\lambda'-\mu_0\in p^fX_0(T)\\
        %0&\text{otherwise},
    %\end{cases}
    %\end{aligned}
    %\end{equation*}
    %where the third isomorphism follows from \cite[II.3.10]{Jan03}. Since $\mu\leq\lambda$, we must have $\mu_1\in C(1)$. Hence the extension group is zero by the Linkage principle (\cite[II.2.12(1),~II.6.17]{Jan03}).

    First we prove that there is an injective $G$-equivariant map $T^{j_0}(\lambda)\into Q_f(\lambda')$. 
    By Proposition \ref{K-ext Prop Tilting Structure}(i), we have an injective $G$-equivariant map $\soc_{G}T(\lambda_{j_0})\cong L(\lambda_{j_0}')\into Q_1(\lambda_{j_0}')$. We claim that this extends to an injective $G$-equivariant map $T(\lambda_{j_0})\into Q_1(\lambda'_{j_0})$. Indeed, it suffices to show that $\Ext^1_{G}\bra{L(\mu),Q_1(\lambda'_{j_0})}=0$ for all $L(\mu)\in\JH(T(\lambda_{j_0}))$, which follows from \cite[Prop.~4.2]{LLHM}. Then, together with the injective $G$-equivariant maps $L(\lambda_{j})\into Q_1(\lambda_j)$ for $j\neq j_0$, we get the desired injective $G$-equivariant map
    \begin{equation*}
        T^{j_0}(\lambda)=T(\lambda_{j_0})^{[j_0]}\otimes\bbra{\otimes_{j\neq j_0}L(\lambda_{j})^{[j]}}\into Q_1(\lambda'_{j_0})^{[j_0]}\otimes\bbra{\otimes_{j\neq j_0}Q_1(\lambda_{j})^{[j]}}\cong Q_f(\lambda'),
    \end{equation*}
    where the last isomorphism follows from \cite[II.11.16]{Jan03}. 

    Next, we separate the following cases.
    \vspace{1em}
    
    \noindent\textbf{Case 1.} $\lambda_{j_0}\in X_1(T)$.
    
    Since $\soc_{G}Q_f(\lambda')\cong L(\lambda')$, we deduce that $\soc_{G}T^{j_0}(\lambda)\cong L(\lambda')$. Since $\lambda_{j_0}\in X_1(T)$, we have $\lambda\in X_f(T)$, hence all the constituents of $T^{j_0}(\lambda)$ are $L(\mu)$ with $\mu\in X_f(T)$. By \cite[Cor.~7.5]{CPSvdK77} (which was originally stated for semisimple groups, and still works for $G$ by the proof \cite[Thm.~4.2.1]{LLHLM20}), we conclude that $\soc_{G(\Fq)}\!\bigbra{T^{j_0}(\lambda)|_{G(\Fq)}}\cong L(\lambda')|_{G(\Fq)}=F(\lambda')$.

    \vspace{1em}
    
    \noindent\textbf{Case 2.} $\lambda_{j_0}\notin X_1(T)$. 
    
    Up to Frobenius twist, we may assume $j_0=f-1$. 

    By [Pil93,~Lemma~4.1] and the proof of \cite[Theorem~4.2.1]{LLHLM20}, we know that $Q_f(\lambda')|_{G(\Fq)}$ is the direct sum of injective envelopes $\op{inj}_{G(\Fq)}F(\mu)$ for $\mu\in X_f(T)$ with multiplicities 
    \begin{equation*}
        \bbbra{Q_f(\lambda')|_{G(\Fq)}:\op{inj}_{G(\Fq)}F(\mu)}=\sum_{\nu\in X(T)_+}\bbbra{L(\mu)\otimes L(\nu):L(p^f\nu+\lambda')}_{G}.
    \end{equation*}
    
    Suppose that $d_{\mu,\nu}\eqdef\bigbbbra{L(\mu)\otimes L(\nu):L(p^f\nu+\lambda')}_{G}\neq0$. Comparing the highest weights, we must have $\mu+\nu\geq p^f\nu+\lambda'$, and we write $\mu=\lambda'+(p^f-1)\nu+c_1\alpha_{12}+c_2\alpha_{23}$ with $c_1,c_2\in\NNN$. Since $\mu\in X_f(T)$, we have
    \begin{align}
        p^f-1\geq\ang{\mu,\alpha_{12}^{\vee}}&=\ang{\lambda',\alpha_{12}^{\vee}}+(p^f-1)\ang{\nu,\alpha_{12}^{\vee}}+2c_1-c_2;\label{K-ext Eq [Q:U] 1}\\
        p^f-1\geq\ang{\mu,\alpha_{23}^{\vee}}&=\ang{\lambda',\alpha_{23}^{\vee}}+(p^f-1)\ang{\nu,\alpha_{23}^{\vee}}-c_1+2c_2.\label{K-ext Eq [Q:U] 2}
    \end{align}
    Since $\lambda_{j_0}\in C(3)\cup C(3')\cup C(p)$, by \eqref{K-ext Eq lambda'} and a case-by-case examination one can check that $\lambda'_{j_0}\in X_1(T)$ with 
    \begin{equation}\label{K-ext Eq [Q:U] 3}
        \ang{\lambda'_{j_0},\alpha_{12}^{\vee}}\geq1;\quad
        \ang{\lambda'_{j_0},\alpha_{23}^{\vee}}\geq1;\quad
        \ang{\lambda'_{j_0},\alpha_{13}^{\vee}}\geq p-1.
    \end{equation}
    For example, if $\lambda_{j_0}\in C(3)$, then we have
    %\begin{equation*}
        $\ang{\lambda'_{j_0},\alpha_{13}^{\vee}}=\ang{s_3(\lambda_{j_0}),\alpha_{13}^{\vee}}=\ang{\lambda_{j_0},\alpha_{23}^{\vee}}+(p-1)\geq p-1.$
    %\end{equation*}
    
    Considering $2\cdot\eqref{K-ext Eq [Q:U] 1}+\eqref{K-ext Eq [Q:U] 2}$ and using $\nu\in X(T)_{+}$, we get
    \begin{equation*}
    \begin{aligned}
        3(p^f-1)&\geq2\ang{\lambda',\alpha_{12}^{\vee}}+\ang{\lambda',\alpha_{23}^{\vee}}+2(p^f-1)\ang{\nu,\alpha_{12}^{\vee}}+(p^f-1)\ang{\nu,\alpha_{23}^{\vee}}+3c_1\\
        &=\ang{\lambda',\alpha_{12}^{\vee}}+\ang{\lambda',\alpha_{13}^{\vee}}+2(p^f-1)\ang{\nu,\alpha_{12}^{\vee}}+(p^f-1)\ang{\nu,\alpha_{23}^{\vee}}+3c_1\\
        &\geq p^{j_0}\bigbra{\ang{\lambda'_{j_0},\alpha_{12}^{\vee}}+\ang{\lambda'_{j_0},\alpha_{13}^{\vee}}}+2(p^f-1)\ang{\nu,\alpha_{12}^{\vee}}+0+0\\
        &\geq p^{f-1}((p-1)+1)+2(p^f-1)\ang{\nu,\alpha_{12}^{\vee}}=p^f+2(p^f-1)\ang{\nu,\alpha_{12}^{\vee}},
    \end{aligned}
    \end{equation*}
    where the last inequality follows from \eqref{K-ext Eq [Q:U] 3} and using $j_0=f-1$. In particular, we must have $\ang{\nu,\alpha_{12}^{\vee}}=0$. Similarly, by considering $\eqref{K-ext Eq [Q:U] 1}+2\cdot\eqref{K-ext Eq [Q:U] 2}$, we also have $\ang{\nu,\alpha_{23}^{\vee}}=0$. Hence we deduce that $\nu\in X_0(T)$, in which case we have $L(\mu)\otimes L(\nu)\cong L(\mu+\nu)$. Therefore, we must have $\mu=\lambda'+(p^f-1)\nu$ (which implies $F(\mu)=F(\lambda')$) and $d_{\mu,\nu}=1$.

    As a consequence, we deduce that $\soc_{G(\Fq)}\!\bigbra{Q(\lambda')|_{G(\Fq)}}\cong F(\lambda')$. Hence we conclude that $\soc_{G(\Fq)}\!\bigbra{T^{j_0}(\lambda)|_{G(\Fq)}}\cong F(\lambda')$.
\end{proof}

\begin{definition}\label{K-ext Def good pair}
    For $\lambda,\mu\in X(T)$, we say that $(\lambda,\mu)$ is a \textbf{good pair} if $\lambda,\mu\in X_1(T)$ and  $(\lambda,\mu)$ belongs to one of the following:
    \begin{equation*}
    \begin{aligned}
        &\bigbra{C(1),C(1)},\bigbra{C(1|2),C(1|2)},\bigbra{C(2|3),C(2|3)},\bigbra{C(2),C(2)},\bigbra{C(1),C(1|2)},\\
        &\bigbra{C(1|2),C(1)},\bigbra{C(1|2),C(2|3)},\bigbra{C(2|3),C(1|2)},\bigbra{C(2|3),C(2)},\bigbra{C(2),C(2|3)}.
    \end{aligned}
    \end{equation*}
    Given $\lambda\in X_1(T)$. For $\mu\in X(T)$ we write 
    \begin{equation*}
        LT^{\lambda}(\mu)\eqdef
    \begin{cases}
        L(\mu)&\text{if}~(\lambda,\mu)~\text{is a good pair}\\
        T(\mu)&\text{otherwise}.
    \end{cases}
    \end{equation*}
\end{definition}

\begin{proposition}\label{K-ext Prop Tensor Structure}
    Let $\lambda\in X_1(T)$. 
\begin{enumerate}
    \item 
    Suppose that $\lambda\in C(1)\cup C(1|2)\cup C(2|3)$.
\begin{enumerate}
    \item 
    If $\lambda|_{\SL_3}\neq(p-1,p-1)$, then we have 
    \begin{equation*}
        L(\lambda)\otimes L(\alpha_{13})\cong\bigbra{\!\oplus_{\alpha\in\Phi^+}LT^{\lambda}(\lambda+\alpha)}\oplus\bigbra{\!\oplus_{\alpha\in\Phi^-}L(\lambda+\alpha)}^{\oplus\delta_{\lambda_1-\lambda_3\leq p-3}}\oplus L(\lambda)^{\oplus\delta_{\lambda}},
    \end{equation*}
    where $\delta_{\lambda}\eqdef\delta_{\lambda_1-\lambda_2\geq1}+\delta_{\lambda_2-\lambda_3\geq1}-\delta_{\lambda_1-\lambda_3=p-3}-\delta_{\lambda_1-\lambda_3=2p-3}$.
    \item 
    If $\lambda|_{\SL_3}=(p-1,p-1)$, then we have
    \begin{equation*}
        L(\lambda)\otimes L(\alpha_{13})\cong T(\lambda+\alpha_{13})\oplus L(\lambda)^{\oplus2}.
    \end{equation*}
\end{enumerate}
    \item 
    Suppose that $\lambda\in C(2)$. Then we have
    \begin{equation*}
    \begin{aligned}
        L(\lambda)\otimes L(\alpha_{13})\cong&\,L(\lambda+\alpha_{13})\oplus LT^{\lambda}(\lambda+\alpha_{12})\oplus LT^{\lambda}(\lambda+\alpha_{23})\\
        &\oplus L(\lambda-\alpha_{13})^{\oplus\delta_{\lambda_1-\lambda_3\geq p+1}}\oplus\bigbra{L(\lambda-\alpha_{12})\oplus L(\lambda-\alpha_{23})}^{\oplus\delta_{\lambda_1-\lambda_3\geq p}}\\
        &\oplus L(\lambda)^{\oplus\bigbra{\delta_{\lambda_1-\lambda_2\leq p-3}+\delta_{\lambda_2-\lambda_3\leq p-3}-\delta_{\lambda_1-\lambda_3=p-1}}}.
    \end{aligned}
    \end{equation*}
\end{enumerate}
\end{proposition}

\begin{proof}
    This follows from a case-by-case examination using the algorithm in \cite[\S7.2]{BDM15} for (i) and \cite[\S7.4]{BDM15} for (ii). We compute one example for (i) and one example for (ii), and leave the other cases as an exercise.
    \vspace{1em}

    \noindent\textbf{Case 1.} $\lambda\in C(1|2)$ and $2\leq\lambda_1-\lambda_2,\lambda_2-\lambda_3\leq p-4$.

    By the Littlewood–Richardson rule, in the Grothendieck group of $G$-modules we have 
    \begin{equation}\label{K-ext Eq Prop Tensor Structure 1}
        [L(\lambda)\otimes L(\alpha_{13})]=[V(\lambda)\otimes V(\alpha_{13})]=\bigbra{\ssum_{\alpha\in\Phi}[V(\lambda+\alpha)]}+2[V(\lambda)].
    \end{equation}
    Since $\lambda+\alpha\in C(2)$ for all $\alpha\in\Phi^+$, by Proposition \ref{K-ext Prop Tilting Structure}(ii) we have
    \begin{equation}\label{K-ext Eq Prop Tensor Structure 2}
    \begin{aligned}
        [T(\lambda+\alpha_{13})]&=[V(\lambda+\alpha_{13})]+[V(\lambda-\alpha_{13})];\\
        [T(\lambda+\alpha_{12})]&=[V(\lambda+\alpha_{12})]+[V(\lambda-\alpha_{23})];\\
        [T(\lambda+\alpha_{23})]&=[V(\lambda+\alpha_{23})]+[V(\lambda-\alpha_{12})].
    \end{aligned}
    \end{equation}
    Combining \eqref{K-ext Eq Prop Tensor Structure 1} and \eqref{K-ext Eq Prop Tensor Structure 2} we get
    \begin{equation}\label{K-ext Eq Prop Tensor Structure 0}
        [L(\lambda)\otimes L(\alpha_{13})]=\bigbra{\ssum_{\alpha\in\Phi^+}[T(\lambda+\alpha)]}+2[V(\lambda)]=\bigbra{\ssum_{\alpha\in\Phi^+}[T(\lambda+\alpha)]}+2[L(\lambda)],
    \end{equation}
    which gives the desired tensor product decomposition by \cite[\S7.2]{BDM15}.
    \vspace{1em}

    \noindent{\textbf{Case 2.}} $\lambda\in C(2)$ and $\lambda|_{\SL_3}=(p-2,1)$. 

    By the Littlewood--Richardson rule, in the Grothendieck group of $G$-modules we have 
    \begin{equation}\label{K-ext Eq Prop Tensor Structure 3}
    \begin{aligned}
        &\bullet~[V(\lambda)\otimes L(\alpha_{13})]=[V(\lambda)\otimes V(\alpha_{13})]=\bigbra{\ssum_{\alpha\in\Phi}[V(\lambda+\alpha)]}+2[V(\lambda)];\\
    &\bullet~\begin{aligned}[t]
        [V(s_{2}(\lambda))\otimes L(\alpha_{13})]&=[V(s_{2}(\lambda))\otimes V(\alpha_{13})]\\
        &=\bigbra{\ssum_{\alpha\in\Phi}[V(s_{2}(\lambda)+\alpha)]}+[V(s_{2}(\lambda))]\\
        &=\bigbra{\ssum_{\alpha\in\Phi}[V(s_{2}(\lambda+\alpha))]}+[V(s_{2}(\lambda))].
    \end{aligned}
    \end{aligned}
    \end{equation}
    Here we recall that $V(\mu)=0$ if $\mu\notin X(T)_+$. Since $\lambda|_{\SL_3}=(p-2,1)$, a case-by-case examination together with the identity $[L(\mu)]=[V(\mu)]-[V(s_2(\mu))]$ for $\mu\in C(2)$ shows that
    \begin{equation}\label{K-ext Eq Prop Tensor Structure 4}
        [V(\mu)]-[V(s_{2}(\mu))]=
    \begin{cases}
        [V(\mu)]&\text{for}~\mu=\lambda+\alpha_{12}\\
        [L(\mu)]&\text{for}~\mu=\lambda+\alpha_{13},\lambda+\alpha_{23},\lambda\\
        0&\text{for}~\mu=\lambda-\alpha_{12},\lambda-\alpha_{23}\\
        -[L(s_{2}(\mu))]&\text{for}~\mu=\lambda-\alpha_{13}.
    \end{cases}
    \end{equation}
    Combining \eqref{K-ext Eq Prop Tensor Structure 3} and \eqref{K-ext Eq Prop Tensor Structure 4} we get
    \begin{equation}\label{K-ext Eq Prop Tensor Structure 5}
    \begin{aligned}
        [L(\lambda)\otimes L(\alpha_{13})]&=[V(\lambda)\otimes L(\alpha_{13})]-[V(s_{2}(\lambda))\otimes L(\alpha_{13})]\\
        &=[L(\lambda+\alpha_{13})]+[V(\lambda+\alpha_{12})]+[L(\lambda+\alpha_{23})]-[L(\lambda)]+[L(\lambda)]+[V(\lambda)].
    \end{aligned}
    \end{equation}
    Moreover, since $\lambda+\alpha_{12}\in C(3)$ and $s_{3}(\lambda+\alpha_{12})=\lambda$, by Proposition \ref{K-ext Prop Tilting Structure}(ii) we have
    \begin{equation}\label{K-ext Eq Prop Tensor Structure 6}
        [T(\lambda+\alpha_{12})]=[V(\lambda+\alpha_{12})]+[V(\lambda)].
    \end{equation}
    Combining \eqref{K-ext Eq Prop Tensor Structure 5} and \eqref{K-ext Eq Prop Tensor Structure 6} we get
    \begin{equation*}
        [L(\lambda)\otimes L(\alpha_{13})]=[L(\lambda+\alpha_{13})]+[T(\lambda+\alpha_{12})]+[L(\lambda+\alpha_{23})],
    \end{equation*}
    which gives the desired tensor product decomposition by \cite[\S7.4]{BDM15}.
\end{proof}

\begin{corollary}\label{K-ext Cor Tensor Socle}
    Let $j_0\in\cJ$ and $\lambda=\sum_{j=0}^{f-1}p^j\lambda_j\in X_f(T)$ with $\lambda_j\in X_1(T)$ for all $j$. Then we have (up to multiplicity)
    \begin{equation*}
        \JH\bigbra{\soc_{G(\Fq)}\bra{F(\lambda)\otimes F(\alpha_{13})^{[j_0]}}}\subseteq\set{F(\lambda+p^{j_0}\alpha):\alpha\in\Phi\cup\set{0},(\lambda_{j_0},\lambda_{j_0}+\alpha)~\text{is a good pair}}.
    \end{equation*}
    Moreover, this inclusion is an equality if $\lambda_{j_0}\notin X_0(T)$.
\end{corollary}

\begin{proof}
    By \cite[Cor.~3.17]{Her09} we have
    \begin{equation*}
        F(\lambda)\otimes F(\alpha_{13})^{[j_0]}\cong\bbra{\otimes_{j=0}^{f-1}F(\lambda_j)^{[j]}}\otimes F(\alpha_{13})^{[j_0]}=\bigbra{F(\lambda_{j_0})\otimes F(\alpha_{13})}^{[j_0]}\otimes\bbra{\otimes_{j\neq j_0}F(\lambda_j)^{[j]}},
    \end{equation*}
    which is the restriction to $G(\Fq)$ of the $G$-representation
    \begin{equation*}
        \bigbra{L(\lambda_{j_0})\otimes L(\alpha_{13})}^{[j_0]}\otimes\bbra{\otimes_{j\neq j_0}L(\lambda_j)^{[j]}}.
    \end{equation*}
    The result then follows from Proposition \ref{K-ext Prop Tensor Structure} and Proposition \ref{K-ext Prop Tensor Socle} with a case-by-case examination for $\lambda_{j_0}$. We compute one case and leave the other cases as an exercise. 

    We suppose that $\lambda_{j_0}\in C(1|2)$ and $2\leq\lambda_{j_0,1}-\lambda_{j_0,2},\lambda_{j_0,2}-\lambda_{j_0,3}\leq p-4$. Then by \eqref{K-ext Eq Prop Tensor Structure 0} we have
    \begin{equation*}
        L(\lambda_{j_0})\otimes L(\alpha_{13})\cong T(\lambda_{j_0}+\alpha_{13})\oplus T(\lambda_{j_0}+\alpha_{12})\oplus T(\lambda_{j_0}+\alpha_{23})\oplus L(\lambda_{j_0})^{\oplus2}.
    \end{equation*}
    Since $\lambda_{j_0}+\alpha\in C(2)$ for all $\alpha\in\Phi^+$, by Proposition \ref{K-ext Prop Tensor Socle} and \eqref{K-ext Eq lambda'} we have 
    \begin{equation*}
    \begin{aligned}
        &\hspace{-3em}\soc_{G(\Fq)}\!\bbra{\!\bigbra{T(\lambda_{j_0}+\alpha_{13})^{[j_0]}\otimes\!\bigbra{\!\otimes_{j\neq j_0}L(\lambda_j)^{[j]}}}|_{G(\Fq)}}\\
        &\cong F\bigbra{s_2(\lambda_{j_0}+\alpha_{13})}^{[j_0]}\otimes\bigbra{\otimes_{j\neq j_0}F(\lambda_j)^{[j]}}\\
        &=F\bra{\lambda_{j_0}-\alpha_{13}}^{[j_0]}\otimes\bigbra{\otimes_{j\neq j_0}F(\lambda_j)^{[j]}}\cong F(\lambda-p^{j_0}\alpha_{13}).
    \end{aligned}
    \end{equation*}
    By a similar computation we also have
    \begin{equation*}
    \begin{aligned}
        %\soc_{G(\Fq)}\!\bbra{\!\bigbra{T(\lambda_{j_0}+\alpha_{13})^{[j_0]}\otimes\!\bigbra{\!\otimes_{j\neq j_0}L(\lambda_j)^{[j]}}}|_{G(\Fq)}}&\cong F(\lambda-p^{j_0}\alpha_{13});\\
        \soc_{G(\Fq)}\!\bbra{\!\bigbra{T(\lambda_{j_0}+\alpha_{12})^{[j_0]}\otimes\!\bigbra{\!\otimes_{j\neq j_0}L(\lambda_j)^{[j]}}}|_{G(\Fq)}}&\cong F(\lambda-p^{j_0}\alpha_{23});\\
        \soc_{G(\Fq)}\!\bbra{\!\bigbra{T(\lambda_{j_0}+\alpha_{23})^{[j_0]}\otimes\!\bigbra{\!\otimes_{j\neq j_0}L(\lambda_j)^{[j]}}}|_{G(\Fq)}}&\cong F(\lambda-p^{j_0}\alpha_{12}).
    \end{aligned}
    \end{equation*}
    Hence the result follows since $\lambda_{j_0}-\alpha\in C(1)$ for all $\alpha\in\Phi^+$.
\end{proof}

\begin{remark}\label{K-ext Rk Calpha+}
    One can write Corollary \ref{K-ext Cor Tensor Socle} in a more concrete manner. We write
    \begin{equation*}
    \begin{aligned}
        C(\alpha_{13})^+&\eqdef\sset{\lambda\in X_1(T)\left|~
        \begin{aligned}
            &0\leq\lambda_1-\lambda_2\leq p-2,~0\leq\lambda_2-\lambda_3\leq p-2\\
            &0\leq\lambda_1-\lambda_3\leq p-4~\text{or}~p-1\leq\lambda_1-\lambda_3\leq 2p-4
        \end{aligned}\right.}\\
        &\hspace{0.6em}\cup\bigset{\lambda\in X_1(T):\lambda|_{\SL_3}=(p-2,0)~\text{or}~(0,p-2)};\\
        C(\alpha_{12})^+&\eqdef\sset{\lambda\in X_1(T)\left|~
        \begin{aligned}
            &0\leq\lambda_1-\lambda_2\leq p-3,~1\leq\lambda_2-\lambda_3\leq p-1\\
            &1\leq\lambda_1-\lambda_3\leq p-3~\text{or}~p-1\leq\lambda_1-\lambda_3\leq 2p-4
        \end{aligned}\right.}\\
        &\hspace{0.6em}\cup\bigset{\lambda\in X_1(T):\lambda|_{\SL_3}=(p-3,1)};\\
        C(\alpha_{23})^+&\eqdef\sset{\lambda\in X_1(T)\left|~
        \begin{aligned}
            &1\leq\lambda_1-\lambda_2\leq p-1,~0\leq\lambda_2-\lambda_3\leq p-3\\
            &1\leq\lambda_1-\lambda_3\leq p-3~\text{or}~p-1\leq\lambda_1-\lambda_3\leq 2p-4
        \end{aligned}\right.}\\
        &\hspace{0.6em}\cup\bigset{\lambda\in X_1(T):\lambda|_{\SL_3}=(1,p-3)};\\
        C(0)^+&\eqdef X_1(T).
    \end{aligned}
    \end{equation*}
    Then for $\lambda\in X_1(T)$ and $\alpha\in\Phi^+\cup\set{0}$, we have 
    \begin{equation*}
        (\lambda,\lambda+\alpha)~\text{is a good pair}\iff(\lambda+\alpha,\lambda)~\text{is a good pair}\iff\lambda\in C(\alpha)^+.
    \end{equation*}
\end{remark}

\section{Algebraic operators on Serre weights}\label{K-ext sec operator Fp}

We keep the notation of \S\ref{K-ext Sec Tensor} and introduce more notation.

We identify the Weyl group of $(G,T)$ with $S_3$ which acts on $X(T)\cong\ZZ^3$ in the usual way. We denote by $w_0\eqdef(13)\in S_3$ the longest element.

Given an algebraic representation $W$ of $G$ and $\mu\in X(T)$, we denote by $W_{\mu}$ the weight space of $W$ associated to the weight $\mu$. For $\lambda\in X(T)_+$, we denote by $v_{\lambda}$ any nonzero element of the $1$-dimensional space $L(\lambda)_{\lambda}$. We also note that $L(\lambda)_{\mu}\neq0$ implies $\mu\leq\lambda$ and $\lambda_3\leq\mu_1,\mu_2,\mu_3\leq\lambda_1$, which will be used throughout this section.

For $\alpha\in\Phi$, we denote by $U_{\alpha}$ the root subgroup of $G$ attached to $\alpha$ and we write $u_{\alpha}:\GG_a\ism U_{\alpha}$ for the isomorphism sending $1\in\GG_a$ to the matrix with $1$ in the entry corresponding to $\alpha$. By \cite[II.1.19(5),(6)]{Jan03} we have
\begin{equation}\label{K-ext Eq Taylor Expansion}
    u_{\alpha}(t)=\sum\limits_{i\geq0}t^iX_{\alpha,i}^{\op{alg}},
\end{equation}
where each $X_{\alpha,i}^{\op{alg}}$ is an element of $\op{Dist}(G)$ (see \cite[II.1.12]{Jan03}) and satisfies $X_{\alpha,i}^{\op{alg}}\bra{W_{\mu}}\subseteq W_{\mu+i\alpha}$ for every $G$-module $W$. For simplicity, we write $X_{\alpha}^{\op{alg}}$ for $X_{\alpha,1}^{\op{alg}}$. Then we have $X_{\alpha,i}^{\op{alg}}=\bigbra{X_{\alpha}^{\op{alg}}}^i/i!$ for $0\leq i\leq p-1$. For simplicity, we write $U_{12}$ for $U_{\alpha_{12}}$, $X_{13}^{\op{alg}}$ for $X_{\alpha_{13}}^{\op{alg}}$, and so on.

\begin{lemma}\label{K-ext Lem Xalg injective}
    Let $\lambda\in X(T)_+$, $\alpha\in\Phi^+$ and $\nu\in X(T)$. Suppose that $\ang{\nu,\alpha^{\vee}}<0$ and $L(\lambda)_{\nu+i\alpha}=0$ for all $i\geq p$ (in particular, this is satisfied if $\nu+p\alpha\nleq\lambda$). Then the map $X_{\alpha}^{\op{alg}}:L(\lambda)_{\nu}\to L(\lambda)_{\nu+\alpha}$ is injective. If moreover $\ang{\nu,\alpha^{\vee}}<-1$, then the map $\bigbra{X_{\alpha}^{\op{alg}}}^2:L(\lambda)_{\nu}\to L(\lambda)_{\nu+2\alpha}$ is also injective. 
\end{lemma}

\begin{proof}
    Let $M$ denote the Levi subgroup $TU_{\alpha}U_{-\alpha}\subseteq G$. By d\'evissage, it suffices to show that for every irreducible constituent $W$ of the restriction $L(\lambda)|_{M}$, the map $X_{\alpha}^{\op{alg}}:W_{\nu}\into W_{\nu+\alpha}$ is injective. If $W_{\nu}=0$, then there is nothing to prove. If $W_{\nu}\neq0$, then we let $\mu\in X(T)$ be the highest weight of $W$ and write $\mu=\nu+i\alpha$ for $i\geq0$. Since $\ang{\nu,\alpha^{\vee}}<0$, we cannot have $i=0$. Since $L(\lambda)_{\mu}\supseteq W_{\mu}\neq0$, we cannot have $i\geq p$. In particular, $i$ is not a multiple of $p$. Then by the proof of the claim in the proof of \cite[Prop.~3.1.2]{HLM17}, we conclude that the map $X_{\alpha}^{\op{alg}}:W_{\nu}\to W_{\nu+\alpha}$ is injective.

    If moreover $\ang{\nu,\alpha^{\vee}}<-1$, then we also cannot have $i=1$, hence $i\not\equiv0,1\,(\mod\,p)$. Then the argument of the proof of the claim in the proof of \cite[Prop.~3.1.2]{HLM17} shows that the map $\bigbra{X_{\alpha}^{\op{alg}}}^2:L(\lambda)_{\nu}\to L(\lambda)_{\nu+2\alpha}$ is also injective.
\end{proof}

\begin{lemma}\label{K-ext Lem L=V}
    Let $\lambda\in X_1(T)$, $\alpha\in\set{\alpha_{12},\alpha_{23}}$ and $\mu\in\set{\lambda-\alpha,\lambda-\ang{\lambda,\alpha^{\vee}}\alpha,\lambda-\ang{\lambda,\alpha^{\vee}}\alpha-\alpha_{13},\lambda-\ang{\lambda,\alpha^{\vee}}\alpha-\ang{\lambda,(\alpha_{13}-\alpha)^{\vee}}\alpha_{13}}$. Then we have $L(\lambda)_{\mu}=V(\lambda)_{\mu}$.
\end{lemma}

\begin{proof}
    If $\lambda\notin C(2)$, then the result is immediate since $L(\lambda)=V(\lambda)$. If $\lambda\in C(2)$, then we have $\lambda_1-\lambda_3\geq p-1$. Using the relation $[L(\lambda)]=[V(\lambda)]-[L(s_2(\lambda))]$ (see \eqref{K-ext Eq reflections} for $s_2$) in the Grothendieck group of $G$-modules, it suffices to show that $L(s_2(\lambda))_{\mu}=0$. Indeed, a case-by-case examination shows that there exists $i\in\set{1,2,3}$ such that $\mu_i=\lambda_1$ or $\lambda_3$. However, we have $(s_2(\lambda))_1=\lambda_3+p-2\leq\lambda_1-1$ and $(s_2(\lambda))_3=\lambda_1-p+2\geq\lambda_3+1$ since $\lambda_1-\lambda_3\geq p-1$, hence we must have $L(s_2(\lambda))_{\mu}=0$.
\end{proof}

\begin{lemma}\label{K-ext Lem nonzero alpha13}
    Let $\lambda\in X_1(T)$ and $\alpha\in\set{\alpha_{12},\alpha_{23}}$. Then we have in $L(\lambda)$
    \begin{equation*}
        \bigbra{X_{31}^{\op{alg}}}^{\ang{\lambda,\bra{\alpha_{13}-\alpha}^{\vee}}}\bigbra{X_{-\alpha}^{\op{alg}}}^{\ang{\lambda,\alpha^{\vee}}}v_{\lambda}\neq0.
    \end{equation*}
\end{lemma}

\begin{proof}
    This follows from \cite[Lemma~2.4.12]{Enn18} together with Lemma \ref{K-ext Lem L=V}.
\end{proof}

\begin{lemma}\label{K-ext Lem nonzero alpha12and23}
    Let $\lambda\in X_1(T)$, $\alpha\in\set{\alpha_{12},\alpha_{23}}$ and $0\neq w\in L(\lambda)_{\lambda-\alpha}$. Then we have in $L(\lambda)$
    \begin{equation*}
        \bigbra{X_{31}^{\op{alg}}}^{\ang{\lambda,\bra{\alpha_{13}-\alpha}^{\vee}}}\bigbra{X_{-\alpha}^{\op{alg}}}^{\ang{\lambda,\alpha^{\vee}}-1}w\neq0.
    \end{equation*}
\end{lemma}

\begin{proof}
    Since $0\neq w\in L(\lambda)_{\lambda-\alpha}$, we have $\ang{\lambda,\alpha^{\vee}}\geq1$. Then by \cite[Lemma~2.4.12]{Enn18} together with Lemma \ref{K-ext Lem L=V}, $w$ must be a scalar multiple of $X_{-\alpha}^{\op{alg}}v_{\lambda}$. Hence we conclude by Lemma \ref{K-ext Lem nonzero alpha13}.
\end{proof}

\begin{lemma}\label{K-ext Lem nonzero alpha=0}
    Let $\lambda\in X_1(T)$, $\alpha\in\set{\alpha_{12},\alpha_{23}}$ and $w\in L(\lambda)_{\lambda-\alpha_{13}}$. Suppose that $X_{\alpha}^{\op{alg}}w\neq0$. Then we have in $L(\lambda)$
    \begin{equation*}
        \bigbra{X_{31}^{\op{alg}}}^{\ang{\lambda,\bra{\alpha_{13}-\alpha}^{\vee}}-1}\bigbra{X_{-\alpha}^{\op{alg}}}^{\ang{\lambda,\alpha^{\vee}}}w\neq0.
    \end{equation*}
\end{lemma}

\begin{proof}
    We give a proof for $\alpha=\alpha_{12}$, the case $\alpha=\alpha_{23}$ being symmetric. 
    
    Since $0\neq X_{12}^{\op{alg}}w\in L(\lambda)_{\lambda-\alpha_{23}}$, we must have $\lambda_2-\lambda_3\geq1$. By \cite[Lemma~2.4.12]{Enn18}, we know that $L(\lambda)_{\lambda-\alpha_{13}}$ is spanned by $X_{21}^{\op{alg}}X_{32}^{\op{alg}}v_{\lambda}$ and $X_{32}^{\op{alg}}X_{21}^{\op{alg}}v_{\lambda}$. We write $w=c_1X_{21}^{\op{alg}}X_{32}^{\op{alg}}v_{\lambda}+c_2X_{32}^{\op{alg}}X_{21}^{\op{alg}}v_{\lambda}$ for $c_1,c_2\in\FF$. 

    Since $X_{12}^{\op{alg}}v_{\lambda}=0$ and $X_{12}^{\op{alg}}X_{32}^{\op{alg}}v_{\lambda}=0$ by weight considerations, using \cite[Lemma~2.4.3]{Enn18} we compute that 
    \begin{equation*}
    \begin{aligned}
        X_{12}^{\op{alg}}X_{21}^{\op{alg}}X_{32}^{\op{alg}}v_{\lambda}&=\ang{\lambda+\alpha_{32},\alpha_{12}^{\vee}}X_{32}^{\op{alg}}v_{\lambda}=(\lambda_1-\lambda_2+1)X_{32}^{\op{alg}}v_{\lambda};\\
        X_{12}^{\op{alg}}X_{32}^{\op{alg}}X_{21}^{\op{alg}}v_{\lambda}&=X_{32}^{\op{alg}}X_{12}^{\op{alg}}X_{21}^{\op{alg}}v_{\lambda}=X_{32}^{\op{alg}}\ang{\lambda,\alpha_{12}^{\vee}}v_{\lambda}=(\lambda_1-\lambda_2)X_{32}^{\op{alg}}v_{\lambda}.
    \end{aligned}
    \end{equation*}
    Hence we deduce from $X_{12}^{\op{alg}}w\neq0$ that $c_1(\lambda_1-\lambda_2+1)+ c_2(\lambda_1-\lambda_2)\neq0$.
    
    Since $[X_{21}^{\op{alg}},X_{32}^{\op{alg}}]=-X_{31}^{\op{alg}}$ which commutes with $X_{21}^{\op{alg}}$ and $X_{32}^{\op{alg}}$, by \cite[Lemma~2.4.2]{Enn18} we have for $m\geq1$
    \begin{equation}\label{K-ext Eq Lem nonzero alpha=0 2}
        \bigbra{X_{21}^{\op{alg}}}^{m}X_{32}^{\op{alg}}=X_{32}^{\op{alg}}\bigbra{X_{21}^{\op{alg}}}^{m}-mX_{31}^{\op{alg}}\bigbra{X_{21}^{\op{alg}}}^{m-1}.
    \end{equation}
    Moreover, we have $\bra{X_{12}^{\op{alg}}}^{\lambda_1-\lambda_2+1}v_{\lambda}=0$ by weight considerations. Hence by \eqref{K-ext Eq Lem nonzero alpha=0 2} we have
    \begin{equation}\label{K-ext Eq Lem nonzero alpha=0 1}
    \begin{aligned}
        \bigbra{X_{21}^{\op{alg}}}^{\lambda_1-\lambda_2}w&=c_1\bigbra{X_{21}^{\op{alg}}}^{\lambda_1-\lambda_2+1}X_{32}^{\op{alg}}v_{\lambda}+c_2\bigbra{X_{21}^{\op{alg}}}^{\lambda_1-\lambda_2}X_{32}^{\op{alg}}X_{21}^{\op{alg}}v_{\lambda}\\
        &=-c_1(\lambda_1-\lambda_2+1)X_{31}^{\op{alg}}\bigbra{X_{21}^{\op{alg}}}^{\lambda_1-\lambda_2}v_{\lambda}-c_2(\lambda_1-\lambda_2)X_{31}^{\op{alg}}\bigbra{X_{21}^{\op{alg}}}^{\lambda_1-\lambda_2}v_{\lambda}\\
        &=-\bigbra{c_1(\lambda_1-\lambda_2+1)+c_2(\lambda_1-\lambda_2)}X_{31}^{\op{alg}}\bigbra{X_{21}^{\op{alg}}}^{\lambda_1-\lambda_2}v_{\lambda}.
    \end{aligned}
    \end{equation}
    By \cite[Lemma~2.4.12]{Enn18} (exchanging the roles of $\alpha$ and $\beta$) together with Lemma \ref{K-ext Lem L=V} we have $\bigbra{X_{31}^{\op{alg}}}^{\lambda_2-\lambda_3}\bigbra{X_{21}^{\op{alg}}}^{\lambda_1-\lambda_2}v_{\lambda}\neq0$. Then multiplying $\bigbra{X_{31}^{\op{alg}}}^{\lambda_2-\lambda_3-1}$ on both sides of  \eqref{K-ext Eq Lem nonzero alpha=0 1} we conclude that $\bigbra{X_{31}^{\op{alg}}}^{\lambda_2-\lambda_3-1}\bra{X_{21}^{\op{alg}}}^{\lambda_1-\lambda_2}w\neq0$.
\end{proof}

\begin{lemma}\label{K-ext Lem Y isomorphism}
    Let $\alpha\in\Phi^+\cup\set{0}$. Let $\lambda\in C(\alpha)^+$ (see Remark \ref{K-ext Rk Calpha+}) such that $\lambda_1-\lambda_3\neq p-2$ if $\alpha\neq\alpha_{13}$ and $\lambda_1-\lambda_2,\lambda_2-\lambda_3\leq p-2$ if $\alpha=0$. Let $\nu=\lambda+\alpha-(p-1)\alpha_{13}$. Then the map $X_{13}^{\op{alg}}:L(\lambda)_{\nu}\to L(\lambda)_{\nu+\alpha_{13}}$ is an isomorphism.
\end{lemma}

\begin{proof}
    First we prove that the map $X_{13}^{\op{alg}}:L(\lambda)_{\nu}\to L(\lambda)_{\nu+\alpha_{13}}$ is injective. 
    
    Suppose that $\alpha=\alpha_{13}$ and $\lambda|_{\SL_3}=(p-2,p-2)$. We write $W(\lambda)$ the dual Weyl module of highest weight $\lambda$. Then the proof of \cite[Prop.~3.1.2]{HLM17} (which still works in the case $a_2=a_1+1=a_0+2$) shows that that map $X_{13}^{\op{alg}}:W(\lambda)_{\nu}\to W(\lambda)_{\nu+\alpha_{13}}=L(\lambda)_{\nu+\alpha_{13}}$ is surjective, with a one-dimensional kernel not contained in $\soc_{G}W(\lambda)=L(\lambda)$. Therefore, the resulting map $X_{13}^{\op{alg}}:L(\lambda)_{\nu}\to L(\lambda)_{\nu+\alpha_{13}}$ is injective.

    In the remaining cases, we have $\ang{\nu,\alpha_{13}^{\vee}}=(\lambda_1-\lambda_3)+\ang{\alpha,\alpha_{13}^{\vee}}-(2p-2)<0$ and $\nu+p\alpha_{13}=\lambda+\alpha+\alpha_{13}>\lambda$. Then the injectivity follows from Lemma \ref{K-ext Lem Xalg injective}.
    \vspace{1em}

    Next we prove that $\dim L(\lambda)_{\nu}=\dim L(\lambda)_{\nu+\alpha_{13}}$. We give a proof for the case $\alpha=\alpha_{12}$, the other cases being similar. In particular, we have $\lambda_1-\lambda_2\leq p-3$.
    
    If $\lambda_1-\lambda_3\leq p-3$, then we have $\nu_3\geq(\nu+\alpha_{13})_3=\lambda_3+0+(p-2)\geq\lambda_1+1$, hence $L(\lambda)_{\nu}=L(\lambda)_{\nu+\alpha_{13}}=0$ and the result follows.
    
    From now on we suppose that $\lambda_1-\lambda_3\geq p-1$. By Kostant multiplicity formula, for $\lambda'\in X(T)_+$ and $\nu'\in X(T)$ we have
    \begin{equation*}
        \dim V(\lambda')_{\nu'}=\sum\limits_{w\in S_3}\op{sgn}(w)p\bigbra{w(\lambda'+\alpha_{13})-(\nu'+\alpha_{13})},
    \end{equation*}
    where $p$ is the partition function given by 
    \begin{equation*}
        p(\mu)=
    \begin{cases}
        \min\set{n_1,n_2}+1&\text{if}~\mu=n_1\alpha_{12}+n_2\alpha_{23}~\text{with}~n_1,n_2\in\NNN\\
        0&\text{otherwise}.
    \end{cases}
    \end{equation*}
    %Hence for $\lambda'$ we have
    %\begin{equation*}
        %\dim W(\lambda)_{\nu}-\dim W(\lambda)_{\nu+\alpha}=\sum\limits_{w\in S_3}\sgn(w)p(w,\lambda)
    %\end{equation*}
    The following table gives $p\bigbra{w(\lambda+\alpha_{13})-(\nu+\alpha_{13})}-p\bigbra{w(\lambda+\alpha_{13})-(\nu+2\alpha_{13})}$ and $p\bigbra{w(s_2(\lambda)+\alpha_{13})-(w_0\nu+\alpha_{13})}-p\bigbra{w(s_2(\lambda)+\alpha_{13})-(w_0(\nu+\alpha_{13})+\alpha_{13})}$ for $w\in S_3$.
    \begin{center}    
    \begin{tabular}{|c|c|c|c|c|c|c|}
        \hline
        %\diagbox{$\lambda'$}{$w$} 
        $w$ & $\id$ & $(12)$ & $(23)$ & $(123)$ & $(132)$ & $(13)$\\ 
        \hline
        $\bra{\lambda,\nu}$ & $1$ & $1$ & $\delta_{\lambda_2-\lambda_3\leq p-2}$ & $0$ & $0$ & $0$\\ 
        \hline
        $\bra{s_2(\lambda),w_0\nu}$ & $1$ & $1$ & $\delta_{\lambda_2-\lambda_3\leq p-2}$ & $0$ & $0$ & $0$\\ 
        \hline
    \end{tabular}
    \end{center}
    %In particular, for $\lambda'\in\set{\lambda,s_2(\lambda)}$ we have
    %\begin{equation*}
        %\dim W(\lambda)_{\nu}-\dim W(\lambda)_{\nu+\alpha}=\sum\limits_{w\in S_3}\sgn(w)\bigbra{p\bigbra{w(\lambda'+\alpha_{13})-(\nu+\alpha_{13})}-p\bigbra{w(\lambda'+\alpha_{13})-(\nu+2\alpha_{13})}}
    %\end{equation*}
    If $\lambda_2-\lambda_3=p-1$, then $\lambda\notin C(2)$, hence $L(\lambda)=V(\lambda)$. Therefore, we have
    \begin{equation*}
        \dim L(\lambda)_{\nu}-\dim L(\lambda)_{\nu+\alpha_{13}}=\dim V(\lambda)_{\nu}-\dim V(\lambda)_{\nu+\alpha_{13}}=1-1-\delta_{\lambda_2-\lambda_3\leq p-2}=0.
    \end{equation*}
    If $\lambda_2-\lambda_3\leq p-2$, then $\lambda\in C(2)$, hence $[L(\lambda)]=[V(\lambda)]-[V(s_{2}(\lambda))]$. Therefore, we have
    \begin{equation*}
    \begin{aligned}
        &\dim L(\lambda)_{\nu}-\dim L(\lambda)_{\nu+\alpha_{13}}\\
        &\hspace{3em}=\bigbra{\dim V(\lambda)_{\nu}-\dim V(\lambda)_{\nu+\alpha_{13}}}-\bigbra{\dim V(s_2(\lambda))_{\nu}-\dim V(s_2(\lambda))_{\nu+\alpha_{13}}}\\
        &\hspace{3em}=\bigbra{\dim V(\lambda)_{\nu}-\dim V(\lambda)_{\nu+\alpha_{13}}}-\bigbra{\dim V(s_2(\lambda))_{w_0\nu}-\dim V(s_2(\lambda))_{w_0(\nu+\alpha_{13})}}=0.
    \end{aligned}
    \end{equation*}
    This completes the proof.
    %\begin{equation*}
    %\begin{aligned}
        %\dim V(\lambda)_{\nu}&=p\bigbra{(p-2)\alpha_{13}}-p\bigbra{(p-2)\alpha_{13}-(a-b+1)\alpha_{12}}-p\bigbra{(p-2)\alpha_{13}-(b-c+1)\alpha_{13}}\\
        %&=(p-1)-(p-2-(a-b))-(p-2-(b-c))=a-c-(p-3).
    %\end{aligned}
    %\end{equation*}
    %\begin{equation*}
    %\begin{aligned}
        %\dim V(\lambda)_{\nu+\alpha_{13}}&=p\bigbra{(p-3)\alpha_{13}}-p\bigbra{(p-3)\alpha_{13}-(a-b+1)\alpha_{12}}-p\bigbra{(p-3)\alpha_{13}-(b-c+1)\alpha_{13}}\\
        %&=(p-2)-(p-3-(a-b))-(p-3-(b-c))
    %\end{aligned}
    %\end{equation*}
\end{proof}

\section{Group operators on Serre weights}\label{K-ext sec operator Fq}

Let $G_0\eqdef\op{Res}_{\Fq/\Fp}\GL_3$ and $\un{G}\eqdef G_0\otimes_{\Fp}\FF$. Then we have $\un{G}\cong\prod_{j=0}^{f-1}\GL_{3/\FF}$ with split maximal torus $\un{T}\cong\prod_{j=0}^{f-1}T$, the group of characters $X(\un{T})\cong X(T)^{\oplus f}$ and $p$-restricted weights $X_1(\un{T})\cong X_1(T)^{\oplus f}$. 

For $\lambda\in X(\un{T})\cong\ZZ^{3f}$, we write $\lambda=\bigbra{\lambda_0,\ldots,\lambda_{f-1}}$ with $\lambda_{j}=\bigbra{\lambda_{j,1},\lambda_{j,2},\lambda_{j,3}}\in X(T)\cong\ZZ^3$. In the case $f=1$, we still write $\lambda=(\lambda_1,\lambda_2,\lambda_3)\in X(T)\cong\ZZ^3$ when there is no possible confusion. For $\alpha\in\Phi$ and $j\in\cJ$, we let $\alpha_{j}\in X(\un{T})$ be $\alpha$ in the $j$-th coordinate and $0$ elsewhere. We let $\un{\alpha}\in X(\un{T})$ be $\alpha$ in every coordinate.

%We identify $X_1(\un{T})$ with $X_f(T)$ via $\lambda\leftrightarrow\wt{\lambda}\eqdef\sum_{j=0}^{f-1}p^j\lambda_{j}$. 

For $\lambda\in X_1(\un{T})$, we identify the $\un{G}$-representation $L(\lambda)$ with $\boxtimes_{j=0}^{f-1}L(\lambda_j)$. Then the restriction $F(\lambda)\eqdef L(\lambda)|_{G_0(\Fp)\cong\GL_3(\Fq)}$ is identified with $\otimes_{j=0}^{f-1}F(\lambda_j)^{[j]}$, and is isomorphic to $F\bigbra{\sum_{j=0}^{f-1}p^j\lambda_j}$ defined in \S\ref{K-ext Sec Tensor}. 

%Then coincides with the $F(\lambda)$. Then we have $F(\lambda)\cong F\bigbra{\sum_{j=0}^{f-1}p^j\lambda_j}$ from $\GL_3$, that is, $\otimes_{j=0}^{f-1}F(\lambda_j)^{[j]}$.

For $\lambda\in X(\un{T})$, we write $\chi_{\lambda}\eqdef\bigbra{\sum_{j=0}^{f-1}p^j\lambda_j}|_{T(\Fq)}:T(\Fq)\to\FF\x$. We write $H\eqdef T([\Fq\x])\subseteq \GL_3(\OL)$ and we still denote by $\chi_{\lambda}$ for the corresponding character of $H$ via the mod $p$ reduction $H\ism T(\Fq)$. Note that $\chi_{\lambda}=\chi_{\mu}$ if and only if $\lambda-\mu\in(p-\pi)X(\un{T})$, where $\pi:X(\un{T})\to X(\un{T})$ is the left shift automorphism $(\lambda_j)_{j=0}^{f-1}\mapsto(\lambda_{j+1})_{j=0}^{f-1}$.

Every representation of $\GL_3(\Fq)$ is also regarded as a representation of $\GL_3(\OL)$ via inflation. Let $\mu\in X(\un{T})$. Given a representation $V$ of $\GL_3(\OL)$, we denote by $V^{\mu}$ the $H$-eigenspace of $V$ with eigencharacter $\chi_{\mu}$. Given an algebraic representation $W$ of $\un{G}$, we denote by $W_{\mu}$ the weight space of $W$ associated to the weight $\mu$. In particular, for $\lambda\in X_1(\un{T})$ we have $L(\lambda)_{\mu}=\otimes_{j=0}^{f-1}L(\lambda_j)_{\mu_j}\subseteq F(\lambda)^{\mu}$. For $\lambda\in X_1(\un{T})$, we denote by $v_{\lambda}$ any nonzero element of the $1$-dimensional space $L(\lambda)_{\lambda}$. We also note that $L(\lambda)_{\mu}\neq0$ implies $\mu_j\leq\lambda_j$ and $\lambda_{j,3}\leq\mu_{j,1},\mu_{j,2},\mu_{j,3}\leq\lambda_{j,1}$ for all $j$, which will be used throughout this section.

For $\alpha\in\Phi$ and $j\in\cJ$, we define (recall that $[t]$ is the Teichm\"uller lift of $t$)
\begin{equation}\label{K-ext Eq Xalpha}
    X_{\alpha,j}\eqdef\sum\limits_{t\in\Fq\x}t^{-p^j}u_{\alpha}([t])\in\FF[U_{\alpha}(\OL)]\subseteq\FF[\GL_3(\OL)].
\end{equation}
For simplicity, we write $X_{13,j}$ for $X_{\alpha_{13},j}$ and so on. In the case $f=1$, we write $X_{\alpha}$ for $X_{\alpha,0}$.

\begin{lemma}\label{K-ext Lem Iwasawa}
    For $\alpha\in\Phi$, we have $\FF\ddbra{U_{\alpha}(\OL)}=\FF\ddbra{X_{\alpha,0},\cdots,X_{\alpha,f-1}}$ and 
    \begin{equation*}
        u_{\alpha}([t])-1\equiv-\ssum_{j=0}^{f-1}t^{p^j}X_{\alpha,j}~\mod\,(X_{\alpha,0},\cdots,X_{\alpha,f-1})^2\quad\forall\,t\in\Fq\x.
    \end{equation*}
\end{lemma}

\begin{proof}
    If $t=1$, this follows as in the Claim in the proof of \cite[Lemma~3.2]{Wang3}. For general $t\in\Fq\x$, we take any $h\in H$ such that $\alpha(h)=[t]$. Then using $hu_{\alpha}([t])=u_{\alpha}\bigbra{\alpha(h)[t]}h$, we have 
    \begin{equation*}
        u_{\alpha}([t])\!-\!1=h\bra{u_{\alpha}(1)\!-\!1}h\inv\equiv-\ssum_{j=0}^{f-1}hX_{\alpha,j}h\inv=-\ssum_{j=0}^{f-1}t^{p^j}X_{\alpha,j}~\mod\,(X_{\alpha,0},\cdots,X_{\alpha,f-1})^2,
    \end{equation*}
    which completes the proof.
\end{proof}

\begin{lemma}\label{K-ext Y changes characters}
    Let $V$ be a representation of $\GL_3(\OL)$. Let $\mu\in X(T)$, $\alpha\in\Phi$ and $j\in\cJ$. Then we have $X_{\alpha,j}\bra{V^{\mu}}\subseteq V^{\mu+\alpha_j}$.
\end{lemma}

\begin{proof}
    This follows from a direct computation using $hu_{\alpha}([t])=u_{\alpha}\bigbra{\alpha(h)[t]}h$ for all $h\in H$ and $t\in\Fq\x$.
\end{proof}

\begin{lemma}\label{K-ext Lem Xalpha injective}
    Let $\lambda\in X_1(\un{T})$.
\begin{enumerate}
    \item 
    Let $\alpha\in\Phi$, $j_0\in\cJ$ and $\mu\in X(\un{T})$. Suppose that the map $X_{\alpha}^{\op{alg}}:L(\lambda_{j_0})_{\mu_{j_0}}\to L(\lambda_{j_0})_{\mu_{j_0}+\alpha}$ is injective. Then the map $X_{\alpha,j_0}:L(\lambda)_{\mu}\to L(\lambda)$ is injective.
    \item 
    More generally, let $\ell\geq1$ and $\alpha^k\in\Phi$, $j_k\in\cJ$, $\mu^k\in X(\un{T})$ for $1\leq k\leq\ell$ such that
\begin{enumerate}
    \item 
    the weights $\mu^k$ are distinct;
    \item 
    the map $X_{\alpha^k}^{\op{alg}}:L(\lambda_{j_k})_{\mu_{j_k}}\to L(\lambda_{j_k})_{\mu_{j_k}+\alpha^{k}}$ is injective for all $k$.
\end{enumerate}    
Let $v\in\oplus_{k=1}^{\ell}L(\lambda)_{\mu^k}$. Suppose that $X_{\alpha^k,j_k}v=0$ for all $k$. Then we have $v=0$.
\end{enumerate}
\end{lemma}

\begin{proof}
    (i). Let $v\in L(\lambda)_{\mu}=\otimes_{j=0}^{f-1}L(\lambda_j)_{\mu_j}$. First we suppose that $v=\otimes_{j=0}^{f-1}v_j$ for some $v_j\in L(\lambda_j)_{\mu_j}$. Then using \eqref{K-ext Eq Taylor Expansion}, \eqref{K-ext Eq Xalpha} and the identification $F(\lambda)=\otimes_{j=0}^{f-1}F(\lambda_j)^{[j]}$, we have  %the computation in $L(\lambda)=\otimes_{j=0}^{f-1}L(\lambda_j)^{[j]}$:
    \begin{equation}\label{K-ext Eq Lem Xalpha injective 0}
    \begin{aligned}
        X_{\alpha,j_0}v&=\ssum_{t\in\Fq}t^{-p^{j_0}}u_{\alpha}(t)v\\
        &=\ssum_{t\in\Fq}t^{-p^{j_0}}\bbra{\otimes_{j=0}^{f-1}u_{\alpha}\bra{t^{p^j}}v_j}\\
        &=\ssum_{t\in\Fq}t^{-p^{j_0}}\ssum_{i_0,\ldots,i_{f-1}\geq0}\bbra{\otimes_{j=0}^{f-1}t^{p^ji_j}X_{\alpha,i_j}^{\op{alg}}v_j}\\
        &=\ssum_{i_0,\ldots,i_{f-1}\geq0}\bbra{\ssum_{t\in\Fq\x}t^{i_0+pi_1+\cdots+p^{f-1}i_{f-1}-p^{j_0}}}\bbra{\otimes_{j=0}^{f-1}X_{\alpha,i_j}^{\op{alg}}v_j}\\
        &=-\ssum_{\un{i}\in D(j_0)}\bbra{\otimes_{j=0}^{f-1}X_{\alpha,i_j}^{\op{alg}}v_j},
    \end{aligned}
    \end{equation}
    where $D(j_0)\eqdef\bigset{\un{i}\in\NNN^f:(q-1)|(i_0+pi_1+\cdots+p^{f-1}i_{f-1}-p^{j_0})}$. Since $\un{i}\in D(j_0)$ implies $\ssum_{j=0}^{f-1}i_j\equiv1\,\mod\,(p-1)$, if we write $X_{\alpha}^{(j_0),\op{alg}}\eqdef\id\boxtimes\cdots\boxtimes X_{\alpha}^{\op{alg}}\boxtimes\cdots\boxtimes\id$, then we have
    \begin{equation}\label{K-ext Eq Lem Xalpha injective 1}
        X_{\alpha,j_0}v\in-X_{\alpha}^{(j_0),\op{alg}}v+\bigbra{\!\oplus_{\nu}L(\lambda)_{\nu}}
    \end{equation}
    with each $\nu$ satisfying $\nu=\mu+\ssum_{i_0,\ldots,i_{f-1}\geq0}i_j\alpha_{j}$ with $\ssum_{j=0}^{f-1}i_j\geq p$. 
    
    In general, by choosing an $\FF$-basis for each $L(\lambda_j)_{\mu_j}$ and writing $v$ as a finite sum of elements of the form $\otimes_{j=0}^{f-1}v_j$, we deduce that \eqref{K-ext Eq Lem Xalpha injective 1} holds for all $v\in L(\lambda)_{\mu}$, which proves (i).

    (ii). The case $\ell=1$ is treated in (i). For $\ell>1$, we let $v\in\oplus_{k=1}^{\ell}L(\lambda)_{\mu^k}$. We choose $k$ such that $\mu^{k}+\alpha^k_{j_k}\neq\mu^{k'}+\ssum_{j=0}^{f-1}i_j\alpha^k_j$ for all $k'\neq k$ and $\un{i}\in D(j_k)$. In particular, this is satisfied if $\ssum_{j=0}^{f-1}\ang{\mu^k_j,\alpha_{13}}$ is minimal (resp.\;maximal) among those $k$ with $\alpha^k\in\Phi^+$ (resp.\;$\alpha^k\in\Phi^-$) since the weights $\mu^k$ are distinct. Then by \eqref{K-ext Eq Lem Xalpha injective 1}, the image of $X_{\alpha^k,j_k}v$ in $L(\lambda)_{\mu^k+\alpha^k_{j_k}}$ is given by $X_{\alpha^k}^{(j_k),\op{alg}}$. Since the map $X_{\alpha^k}^{\op{alg}}:L(\lambda_{j_k})_{\mu_{j_k}}\to L(\lambda_{j_k})_{\mu_{j_k}+\alpha^{k}}$ is injective and $X_{\alpha^k,j_k}v=0$, we deduce that $v$ has zero image in $L(\lambda)_{\mu^k}$. Hence we conclude that $v=0$ by induction hypothesis.
\end{proof}

\begin{lemma}\label{K-ext Lem Y maps into}
    Let $\lambda\in X_1(\un{T})$, $\alpha\in\Phi$ and $\mu\in X(\un{T})$. Suppose that $L(\lambda_j)_{\mu_j+i\alpha}=0$ for all $i\geq p$ and $j\in\cJ$. Then we have $X_{\alpha,j}=-X_{\alpha}^{(j),\op{alg}}$ on $L(\lambda)_{\mu}$ for all $j$ (see above \eqref{K-ext Eq Lem Xalpha injective 1} for $X_{\alpha}^{(j),\op{alg}}$).
\end{lemma}

\begin{proof}
    Let $j_0\in\cJ$. If $0\leq i_j\leq p-1$ satisfying $(q-1)|(i_0+pi_0+\cdots+p^{f-1}i_{f-1}-p^{j_0})$, then we must have $i_j=\delta_{j=j_0}$ for all $j$. The result then follows from \eqref{K-ext Eq Lem Xalpha injective 0}.
\end{proof}

\begin{lemma}\label{K-ext Lem operator mapsto zero}
    Let $\lambda\in X_1(\un{T})$, $j_0\in\cJ$ and $\alpha\in\set{\alpha_{21},\alpha_{32}}$. Let $\un{a},\un{b}\in\NNN^f$ such that
\begin{enumerate}
    \item 
    $\sum_{j\neq j_0}(a_j+b_j)\geq\bigbra{\!\sum_{j\neq j_0}(\lambda_{j,1}-\lambda_{j,3})}-(p-1)$;
    \item 
    $a_{j_0}+b_{j_0}\geq\lambda_{j_0,1}-\lambda_{j_0,3}+1$.
\end{enumerate}
    Then we have $\bigbra{\!\prod_{j=0}^{f-1}{X_{31}^{a_j}X_{\alpha}^{b_j}}}v_{\lambda}=0$ in $L(\lambda)$.
\end{lemma}

\begin{proof}
    We give a proof for $\alpha=\alpha_{21}$, the case $\alpha=\alpha_{32}$ being symmetric. By \eqref{K-ext Eq Lem Xalpha injective 1} we have
    \begin{equation*}
        \bbra{\sprod_{j=0}^{f-1}{X_{31}^{a_j}X_{\alpha}^{b_j}}}v_{\lambda}\in L(\lambda)_{\mu}\oplus\bigbra{\!\oplus_{\nu}L(\lambda)_{\nu}},
    \end{equation*}
    where $\mu=\lambda+\ssum_{j=0}^{f-1}\bigbra{a_j\alpha_{31}+b_j\alpha_{21}}$, and each $\nu$ satisfies $\nu=\lambda+\ssum_{j=0}^{f-1}\bigbra{a'_j\alpha_{31}+b'_j\alpha_{21}}$ with $\un{a}',\un{b}'\in\NNN^f$ and $\ssum_{j=0}^{f-1}\bigbra{a'_j+b'_j}\geq\ssum_{j=0}^{f-1}\bra{a_j+b_j}+(p-1)\geq\ssum_{j=0}^{f-1}\bra{\lambda_{j,1}-\lambda_{j,3}}+1$.

    Since $\mu_{j_0,1}=\lambda_{j_0,1}-a_{j_0}-b_{j_0}\leq\lambda_{j_0,3}-1$, we have $L(\lambda)_{\mu}=0$. Also, for each $\nu$ as above, there exists $j$ such that $a'_j+b'_j\geq\lambda_{j,1}-\lambda_{j,3}+1$, which implies $\nu_{j,1}=\lambda_{j,1}-a'_j-b'_j\leq\lambda_{j,3}-1$, hence $L(\lambda)_{\nu}=0$. Therefore, we conclude that $\bigbra{\!\prod_{j=0}^{f-1}{X_{31}^{a_j}X_{\alpha}^{b_j}}}v_{\lambda}=0$.
\end{proof}

\begin{lemma}\label{K-ext Lem lambda-alpha12}
    Let $j_0\in\cJ$ and $\lambda\in X_1(\un{T})$ such that $\lambda_{j_0-1,1}-\lambda_{j_0-1,3}\leq 2p-3$. Let $\alpha\in\set{\alpha_{12},\alpha_{23}}$ and $w\in F(\lambda)^{\lambda-\alpha_{j_0}}$. Suppose that $X_{\alpha_{13}-\alpha,j_0}w=X_{13,j_0-1}w=0$. Then we have $w\in L(\lambda)_{\lambda-\alpha_{j_0}}$.
\end{lemma}

\begin{proof}
    We give a proof for $\alpha=\alpha_{12}$, the case $\alpha=\alpha_{23}$ being symmetric. We write $\mu\eqdef\lambda-\alpha_{12,j_0}$.
    
    Let $\nu\in X(\un{T})$ such that $0\neq L(\lambda)_{\nu}\subseteq F(\lambda)^{\mu}$. Then we have 
    \begin{equation*}
        \nu-\mu\in(p-\pi)X(\un{T})\cap\bbbra{\oplus_{j=0}^{f-1}\bigbra{\ZZ\alpha_{12,j}\oplus\ZZ\alpha_{23,j}}}=\oplus_{j=0}^{f-1}\bigbbbra{\ZZ(p-\pi)\alpha_{12,j}\oplus\ZZ(p-\pi)\alpha_{23,j}}.
    \end{equation*}
    Hence there exists $\un{c}_1,\un{c}_2\in\ZZ^f$ such that
    \begin{equation}\label{K-ext Eq Lem lambda-alpha12 1}
        \nu=\lambda-\alpha_{12,j_0}-\ssum_{j=0}^{f-1}\bigbra{\bra{pc_{1,j}-c_{1,j-1}}\alpha_{12,j}+\bra{pc_{2,j}-c_{2,j-1}}\alpha_{23,j}}.
    \end{equation}
    
    Since $\lambda_{j,3}\leq\nu_{j,1}\leq\lambda_{j,1}$ for all $j$, we must have
    %\begin{equation*}
    %\begin{aligned}
        %&-1\leq pc_{1,j_0}-c_{1,j_0-1}\leq\lambda_{j_0,1}-\lambda_{j_0,3}-1\leq2p-3;\\
        %&0\leq pc_{1,j}-c_{1,j-1}\leq\lambda_{j,1}-\lambda_{j,3}\leq 2p-2-\delta_{j=j_0-1}\quad\forall\,j\neq{j_0}.
    %\end{aligned}
    %\end{equation*}
    \begin{equation*}
        -\delta_{j=j_0}\leq pc_{1,j}-c_{1,j-1}\leq\lambda_{j,1}-\lambda_{j,3}-\delta_{j=j_0}\leq 2p-2-\delta_{j=j_0}\quad\forall\,j.
    \end{equation*}
    Using $c_{1,j}=\bigbbbra{\sum_{i=0}^{f-1}p^{f-1-i}\bra{pc_{1,j-i}-c_{1,j-i-1}}}/(q-1)$, we deduce that $0\leq c_{1,j}\leq1$ for all $j$. Since $pc_{1,j}-c_{1,j-1}\geq0$ for all $j\neq j_0$, we have $c_{1,j}=0$ implies $c_{1,j-1}=0$ for all $j\neq j_0$.

    Since $\lambda_{j,3}\leq\nu_{j,3}\leq\lambda_{j,1}$ for all $j$ and $\lambda_{j_0-1,1}-\lambda_{j_0-1,3}\leq 2p-3$, we must have
    \begin{equation*}
        0\leq pc_{2,j}-c_{2,j-1}\leq\lambda_{j,1}-\lambda_{j,3}\leq(2p-2)-\delta_{j=j_0-1}\quad\forall\,j.
    \end{equation*}
    As in the previous paragraph, we have $0\leq c_{2,j}\leq1$ for all $j$. Since $pc_{2,j}-c_{2,j-1}\geq0$ for all $j$, we must have either $\un{c}_2=\un{1}$ or $\un{c}_2=\un{0}$. 

    Since $\nu_{j,2}\leq\lambda_{j,1}$ for all $j$, we must have
    \begin{equation*}
        p\bra{c_{1,j}-c_{2,j}}-\bra{c_{1,j-1}-c_{2,j-1}}\leq\lambda_{j,1}-\lambda_{j,2}-\delta_{j=j_0}\leq(p-1)-\delta_{j=j_0}\quad\forall\,j,
    \end{equation*}
    which implies $c_{1,j}-c_{2,j}\leq0$ for all $j$.
    
    %Similarly, from $\lambda_{j,3}\leq\nu_{j,3}\leq\lambda_{j,1}$ for all $j$ and using $\lambda_{j_0-1,1}-\lambda_{j_0-1,3}\leq 2p-3$, we deduce that either $\un{c}_2=\un{1}$ or $\un{c}_2=\un{0}$. From $\lambda_{j,3}\leq\nu_{j,2}\leq\lambda_{j,1}$ we deduce that $-1\leq c_{1,j}-c_{2,j}\leq0$ for all $j\in\cJ$.

    For $j_1\in\cJ$, we let $\nu^{j_1}\in X_1(\un{T})$ be as in \eqref{K-ext Eq Lem lambda-alpha12 1} with $\un{c}_2=\un{1}$, $c_{1,j}=1$ for $j=j_1,j_1+1,\ldots,j_0-1$ and $c_{1,j}=0$ otherwise. More explicitly, we have
    \begin{equation*}
        \nu^{j_1}_j=
    \begin{cases}
        \lambda_j-(p-1+\delta_{j=j_1})\alpha_{12}-(p-1)\alpha_{23}&\text{if}~j=j_1,j_1+1,\ldots,j_0-1\\
        \lambda_j-(p-1)\alpha_{23}&\text{otherwise}.
    \end{cases}
    \end{equation*}
    We also let $\nu^{\emptyset}\eqdef\lambda-\alpha_{12,j_0}-(p-1)\un{\alpha_{23}}$. Then the previous paragraphs show that
    \begin{equation*}
        w\in L(\lambda)_{\mu}\oplus L(\lambda)_{\nu^{\emptyset}}\oplus\bigbra{\!\oplus_{j_1\in\cJ}L(\lambda)_{\nu^{j_1}}}.
    \end{equation*}

    For each $j_1\in\cJ$, since $\ang{\nu^{j_1}_{j_0-1},\alpha_{13}^{\vee}}\leq\ang{\lambda_{j_0-1}-(p-1)\alpha_{13},\alpha_{13}^{\vee}}=\lambda_{j_0-1,1}-\lambda_{j_0-1,3}-(2p-2)<0$, and $\nu^{j_1}_{j_0-1}+i\alpha_{13}\nleq\lambda_{j_0-1}$ for $i\geq p$, the map $X_{13}^{\op{alg}}:L(\lambda_{j_0-1})_{\nu^{j_1}_{j_0-1}}\to L(\lambda_{j_0-1})_{\nu^{j_1}_{j_0-1}+\alpha_{13}}$ is injective by Lemma \ref{K-ext Lem Xalg injective}. Similarly, the map $X_{23}^{\op{alg}}:L(\lambda_{j_0})_{\nu^{\emptyset}_{j_0}}\to L(\lambda_{j_0})_{\nu^{\emptyset}_{j_0}+\alpha_{23}}$ is injective by Lemma \ref{K-ext Lem Xalg injective}. Moreover, by Lemma \ref{K-ext Lem Y maps into} we have $X_{13,j_0-1}|_{L(\lambda)_{\mu}}=X_{23,j_0}|_{L(\lambda)_{\mu}}=0$. Then by Lemma \ref{K-ext Lem Xalpha injective} (ii) applied to $\bigset{\bra{\alpha_{23},j_0,\nu^{\emptyset}},\bra{\alpha_{13},j_0-1,\nu^{j_1}}:j_1\in\cJ}$ we conclude that $w\in L(\lambda)_{\mu}$.
\end{proof}

\begin{lemma}\label{K-ext Lem lambda-alpha13}
    Let $j_0\in\cJ$ and $\lambda\in X_1(\un{T})$ such that $\lambda_{j_0-1,1}-\lambda_{j_0-1,2},\lambda_{j_0-1,2}-\lambda_{j_0-1,3}\leq p-2$. Let $w\in F(\lambda)^{\lambda-\alpha_{13,j_0}}$ and $\alpha\in\set{\alpha_{12},\alpha_{23}}$. Suppose that $X_{\alpha,j_0-1}^{1+\delta_{f=1}}w=0$. Then we have $w\in L(\lambda)_{\lambda-\alpha_{13,j_0}}$.
\end{lemma}

\begin{proof}
    We give a proof for $\alpha=\alpha_{12}$, the case $\alpha=\alpha_{23}$ being symmetric. We write $\mu\eqdef\lambda-\alpha_{13,j_0}$.
    
    Let $\nu\in X(\un{T})$ such that $0\neq L(\lambda)_{\nu}\subseteq F(\lambda)^{\lambda-\alpha_{13,j_0}}$. Then as in \eqref{K-ext Eq Lem lambda-alpha12 1} there exists $\un{c}_1,\un{c}_2\in\ZZ^f$ such that
    \begin{equation}\label{K-ext Eq Lem lambda-alpha13 1}
        \nu=\lambda-\alpha_{13,j_0}-\ssum_{j=0}^{f-1}\bigbra{\bra{pc_{1,j}-c_{1,j-1}}\alpha_{12,j}+\bra{pc_{2,j}-c_{2,j-1}}\alpha_{23,j}}.
    \end{equation}
    Since $\lambda_{j,3}\leq\nu_{j,1},\nu_{j,2},\nu_{j,3}\leq\lambda_{j,1}$ for all $j$ and using $\lambda_{j_0-1,1}-\lambda_{j_0-1,2},\lambda_{j_0-1,2}-\lambda_{j_0-1,3}\leq p-2$, a similar analysis as in the proof of Lemma \ref{K-ext Lem lambda-alpha12} shows that $0\leq c_{1,j}=c_{2,j}\leq1$ for all $j$, and $c_{1,j}=0$ implies $c_{1,j-1}=0$ for all $j\neq j_0$. 

    For $j_1\in\cJ$, we let $\nu^{j_1}\in X_1(\un{T})$ be as in \eqref{K-ext Eq Lem lambda-alpha13 1} with $c_{1,j}=c_{2,j}=1$ for $j=j_1,j_1+1,\ldots,j_0-1$ and $c_{1,j}=c_{2,j}=0$ otherwise. More explicitly, we have
    \begin{equation*}
        \nu^{j_1}_j=
    \begin{cases}
        \lambda_j-(p-1+\delta_{j=j_1})\alpha_{13}&\text{if}~j=j_1,j_1+1,\ldots,j_0-1\\
        \lambda_j&\text{otherwise}
    \end{cases}
    \end{equation*}
    Then the previous paragraph shows that
    \begin{equation*}
        w\in L(\lambda)_{\mu}\oplus\bigbra{\!\oplus_{j_1\in\cJ}L(\lambda)_{\nu^{j_1}}}.
    \end{equation*}

    First we suppose that $f\geq2$. Let $j_1\in\cJ$. Then we have 
    \begin{equation*}
        \ang{\nu^{j_1}_{j_0-1},\alpha_{12}^{\vee}}\leq\ang{\lambda_{j_0-1}-(p-1)\alpha_{13},\alpha_{12}^{\vee}}=\lambda_{j_0-1,1}-\lambda_{j_0-1,2}-(p-1)<0.
    \end{equation*}
    We also have $L(\lambda_{j_0-1})_{\nu^{j_1}_{j_0-1}+i\alpha_{12}}=0$ for $i\geq p$ since $\bigbra{\nu^{j_1}_{j_0-1}+i\alpha_{12}}_2=\lambda_{j_0-1,2}-i\leq\lambda_{j_0-1,3}-1$. Hence the map $X_{12}^{\op{alg}}:L(\lambda_{j_0-1})_{\nu^{j_1}_{j_0-1}}\to L(\lambda_{j_0-1})_{\nu^{j_1}_{j_0-1}+\alpha_{12}}$ is injective by Lemma \ref{K-ext Lem Xalg injective}. Moreover, by Lemma \ref{K-ext Lem Y maps into} and using $f\geq2$ (hence $j_0-1\neq j_0$) we have $X_{12,j_0-1}|_{L(\lambda)_{\mu}}=0$. Then by Lemma \ref{K-ext Lem Xalpha injective}(ii) applied to $\bigset{\bra{\alpha_{13},j_0-1,\nu^{j_1}}:j_1\in\cJ}$ we conclude that $w\in L(\lambda)_{\mu}$.

    In the case $f=1$, the map $\bigbra{X_{12}^{\op{alg}}}^2:L(\lambda)_{\nu^0}\to L(\lambda)_{\nu^0+2\alpha_{12}}$ is injective by the second part of Lemma \ref{K-ext Lem Xalg injective}. Moreover, by Lemma \ref{K-ext Lem Y maps into} we have $X_{12}^2|_{L(\lambda)_{\mu}}=0$. Hence we conclude that $w\in L(\lambda)_{\mu}$ as in the proof of Lemma \ref{K-ext Lem Xalpha injective}(i).
\end{proof}

\begin{definition}\label{K-ext Def Calpha}
    Suppose that $f=1$. We define the following subsets of $X_1(T)$ (see Remark \ref{K-ext Rk Calpha+} for $C(\alpha)^+$):
    \begin{equation*}
    \begin{aligned}
        C(\alpha_{13},0)&\eqdef\bigset{\lambda\in C(\alpha_{13})^+:\min\set{\lambda_1-\lambda_2,\lambda_2-\lambda_3}\leq p-4};\\
        C(\alpha_{12},0)&\eqdef\bigset{\lambda\in C(\alpha_{12})^+:\lambda_1-\lambda_2\leq p-4};\\
        C(\alpha_{23},0)&\eqdef\bigset{\lambda\in C(\alpha_{23})^+:\lambda_2-\lambda_3\leq p-4};\\
        C(0,0)&\eqdef\bigset{\lambda\in X_1(T):\lambda_1-\lambda_2\leq p-3,~\lambda_2-\lambda_3\leq p-3~\text{and}~\lambda_1-\lambda_3\neq p-2}.
    \end{aligned}
    \end{equation*}
\end{definition}

\begin{definition}\label{K-ext Def Calphaj0}
    Suppose that $f\geq2$. Let $j_0\in\cJ$. We define the following subsets of $X_1(\un{T})$:
    \begin{equation*}
    \begin{aligned}
        C(\alpha_{13},j_0)&\eqdef\bigset{\lambda\in X_1(\un{T}):\lambda_{j_0}\in C(\alpha_{13})^+~\text{and}~\lambda_{j_0-1,1}-\lambda_{j_0-1,3}\leq 2p-3};\\
        C(\alpha_{12},j_0)&\eqdef\sset{\lambda\in X_1(\un{T}):\lambda_{j_0}\in C(\alpha_{12})^+~\text{and}~\lambda_{j_0-1,1}-\lambda_{j_0-1,2}\leq p-2}\\
        &\hspace{3em}-\bigset{\lambda\in X_1(\un{T}):\lambda_{j_0}|_{\SL_3}=(p-3,1)~\text{and}~\lambda_{j,1}-\lambda_{j,3}\geq p-1~\forall\,j\neq j_0};\\
        C(\alpha_{23},j_0)&\eqdef\sset{\lambda\in X_1(\un{T}):\lambda_{j_0}\in C(\alpha_{23})^+~\text{and}~\lambda_{j_0-1,2}-\lambda_{j_0-1,3}\leq p-2}\\
        &\hspace{3em}-\bigset{\lambda\in X_1(\un{T}):\lambda_{j_0}|_{\SL_3}=(1,p-3)~\text{and}~\lambda_{j,1}-\lambda_{j,3}\geq p-1~\forall\,j\neq j_0};\\
        C(0,j_0)&\eqdef\sset{\lambda\in X_1(\un{T}):\lambda_{j,1}-\lambda_{j,2}\leq p-2~\text{and}~\lambda_{j,2}-\lambda_{j,3}\leq p-2~\forall\,j}.
    \end{aligned}
    \end{equation*}
\end{definition}

\begin{lemma}\label{K-ext Lem character-weight same}
    Let $\alpha\in\Phi^+\cup\set{0}$, $j_0\in\cJ$ and $\lambda\in C(\alpha,j_0)$. Let $\mu=\lambda+\alpha_{j_0}+\beta_{j_0}$ with $\beta\in\Phi^+$. Then we have $F(\lambda)^{\mu}=L(\lambda)_{\mu-(p-1)\un{\alpha_{13}}}$.
\end{lemma}

\begin{proof}
    We give a proof in the case $f\geq2$, the case $f=1$ being similar and simpler. 
    
    Let $\nu\in X(\un{T})$ such that $0\neq L(\lambda)_{\nu}\subseteq F(\lambda)^{\mu}$. Then as in \eqref{K-ext Eq Lem lambda-alpha12 1} there exists $\un{c}_1,\un{c}_2\in\ZZ^f$ such that
    \begin{equation}\label{K-ext Eq Lem character-weight same 1}
        \nu=\mu-\ssum_{j=0}^{f-1}\bigbra{\bra{pc_{1,j}-c_{1,j-1}}\alpha_{12,j}+\bra{pc_{2,j}-c_{2,j-1}}\alpha_{23,j}}.
    \end{equation}
    
    Since $\lambda_{j,3}\leq\nu_{j,1}\leq\lambda_{j,1}$ for all $j$, we must have (using $\lambda\in C(\alpha,j_0)$)
    \begin{equation}\label{K-ext Eq Lem character-weight same 2}
    \begin{aligned}
        &0\leq\alpha_{1}+\beta_{1}\leq pc_{1,j_0}-c_{1,j_0-1}\leq\bra{\lambda_{j_0,1}-\lambda_{j_0,3}+\alpha_{1}}+\beta_{1}\leq(2p-3)+1=2p-2;\\
        &0\leq pc_{1,j}-c_{1,j-1}\leq\lambda_{j,1}-\lambda_{j,3}\leq 2p-2-\delta_{j=j_0-1}\quad\forall\,j\neq{j_0}.
    \end{aligned}
    \end{equation}
    Using $c_{1,j}=\bigbbbra{\sum_{i=0}^{f-1}p^{f-1-i}\bra{pc_{1,j-i}-c_{1,j-i-1}}}/(q-1)$, we deduce that $0\leq c_{1,j}\leq1$ for all $j$. Since $pc_{1,j}-c_{1,j-1}\geq0$ for all $j$, we deduce that either $\un{c}_1=\un{1}$ or $\un{c}_1=\un{0}$.

    Similarly, from $\lambda_{j,3}\leq\nu_{j,3}\leq\lambda_{j,1}$ for all $j$ we must have (using $\lambda\in C(\alpha,j_0)$)
    \begin{equation}\label{K-ext Eq Lem character-weight same 3}
    \begin{aligned}
        &0\leq-\alpha_{3}-\beta_{3}\leq pc_{2,j_0}-c_{2,j_0-1}\leq\bra{\lambda_{j_0,1}-\lambda_{j_0,3}-\alpha_{3}}-\beta_{3}\leq(2p-3)+1=2p-2;\\
        &0\leq pc_{2,j}-c_{2,j-1}\leq\lambda_{j,1}-\lambda_{j,3}\leq 2p-2-\delta_{j=j_0-1}\quad\forall\,j\neq{j_0},
    \end{aligned}
    \end{equation}
    which implies either $\un{c}_2=\un{1}$ or $\un{c}_2=\un{0}$.

    If $\un{c}_1=\un{0}$, then by \eqref{K-ext Eq Lem character-weight same 2} we have $\alpha_{1}=0$, i.e.\;$\alpha=\alpha_{23}$ or $0$, hence $\lambda_{j_0-1,2}-\lambda_{j_0-1,3}\leq p-2$. Since $\nu_{j_0-1,2}\geq\lambda_{j_0-1,3}$, by \eqref{K-ext Eq Lem character-weight same 1} and using $f\geq2$ (hence $j_0-1\neq j_0$) we have 
    \begin{equation*}
        p\bra{c_{1,j_0-1}-c_{2,j_0-1}}-\bra{c_{1,j_0-2}-c_{2,j_0-2}}\geq-\bra{\lambda_{j_0-1,2}-\lambda_{j_0-1,3}}\geq-(p-2).
    \end{equation*}
    Hence we must have $\un{c}_2=\un{0}$ and $\nu=\mu$. This is a contradiction since $L(\lambda)_{\mu}=0$. %Similarly, $\un{c}_2=\un{0}$ also leads to a contradiction.

    Similarly, if $\un{c}_2=\un{0}$, then by \eqref{K-ext Eq Lem character-weight same 3} and $\nu_{j_0-1,2}\leq\lambda_{j_0-1,1}$ we have $\alpha=\alpha_{12}$ or $0$, and
    \begin{equation*}
        p\bra{c_{1,j_0-1}-c_{2,j_0-1}}-\bra{c_{1,j_0-2}-c_{2,j_0-2}}\leq\lambda_{j_0-1,1}-\lambda_{j_0-1,2}\leq(p-2).
    \end{equation*}
    Hence we must have $\un{c}_1=\un{0}$ and $\nu=\mu$. This is a contradiction since $L(\lambda)_{\mu}=0$.

    Hence the only possibility is $\un{c}_1=\un{c}_2=1$ (if such $\nu$ exists), in which case $\nu=\mu-(p-1)\un{\alpha_{13}}$. We conclude that $F(\lambda)^{\mu}=L(\lambda)_{\mu-(p-1)\un{\alpha_{13}}}$.
\end{proof}

\begin{lemma}\label{K-ext Lem character-weight different}
    Let $\alpha\in\Phi^+$, $j_0\in\cJ$ and $\lambda\in C(\alpha,j_0)$. Let $\mu=\lambda+\alpha_{j_0}+\beta_{j_1}$ with $\beta\in\Phi^+$ and $j_1\neq j_0$. Then $F(\lambda)^{\mu}$ is a direct sum of weight spaces $L(\lambda)_{\nu}$ with $\nu_{j_0}=\lambda_{j_0}+\alpha-(p-1)\alpha_{13}$.
\end{lemma}

\begin{proof}
    Let $\nu\in X(\un{T})$ such that $0\neq L(\lambda)_{\nu}\subseteq F(\lambda)^{\mu}$. Then as in \eqref{K-ext Eq Lem lambda-alpha12 1} there exists $\un{c}_1,\un{c}_2\in\ZZ^f$ such that
    \begin{equation}\label{K-ext Eq Lem character-weight different 1}
        \nu=\mu-\ssum_{j=0}^{f-1}\bigbra{\bra{pc_{1,j}-c_{1,j-1}}\alpha_{12,j}+\bra{pc_{2,j}-c_{2,j-1}}\alpha_{23,j}}.
    \end{equation}
    
    Since $\lambda_{j,3}\leq\nu_{j,1}\leq\lambda_{j,1}$ for all $j$, we must have (using $\lambda\in C(\alpha,j_0)$)
    \begin{equation*}
    \begin{aligned}
        &0\leq\alpha_{1}\leq pc_{1,j_0}-c_{1,j_0-1}\leq\lambda_{j_0,1}-\lambda_{j_0,3}+\alpha_{1}\leq 2p-3;\\
        &0\leq\beta_{1}\leq pc_{1,j_1}-c_{1,j_1-1}\leq\bra{\lambda_{j_1,1}-\lambda_{j_1,3}}+\beta_{1}\leq 2p-1-\delta_{j_1=j_0-1};\\
        &0\leq pc_{1,j}-c_{1,j-1}\leq\lambda_{j,1}-\lambda_{j,3}\leq 2p-2-\delta_{j=j_0-1}\quad\forall\,j\neq{j_0,j_1}.
    \end{aligned}
    \end{equation*}
    Using $c_{1,j}=\bigbbbra{\sum_{i=0}^{f-1}p^{f-1-i}\bra{pc_{1,j-i}-c_{1,j-i-1}}}/(q-1)$ and $pc_{1,j}-c_{1,j-1}\geq0$ for all $j$, we deduce that either $\un{c}_1=\un{0}$ (which implies $\alpha_1=0$, hence $\alpha=\alpha_{23}$), or $1\leq c_{1,j}\leq2$ for all $j$ and $c_{1,j_0}=c_{1,j_0-1}=1$.

    Similarly, from $\lambda_{j,3}\leq\nu_{j,3}\leq\lambda_{j,1}$ for all $j$ we deduce that either $\un{c}_2=\un{0}$ (which implies $\alpha=\alpha_{12}$), or $1\leq c_{1,j}\leq2$ for all $j$ and $c_{1,j_0}=c_{1,j_0-1}=1$.

    If $\un{c}_{1}=\un{0}$, then $\alpha=\alpha_{23}$, and from $\nu_{j_0-1,2}\geq\lambda_{j_0-1,3}$ we have (using $\lambda\in C(\alpha,j_0)$)
    \begin{equation*}
        p\bra{c_{1,j_0}-c_{2,j_0}}-\bra{c_{1,j_0-1}-c_{2,j_0-1}}\geq-1-\bra{\lambda_{j_0,2}-\lambda_{j_0,3}}\geq-1-(p-3)=-(p-2).
    \end{equation*}
    Hence we must have $\un{c}_2=\un{0}$ and $\nu=\mu$. This is a contradiction since $L(\lambda)_{\mu}=0$. Similarly, $\un{c}_2=\un{0}$ also leads to a contradiction.

    Hence we must have $c_{1,j_0}=c_{1,j_0-1}=c_{2,j_0}=c_{2,j_0-1}=1$ (if such $\nu$ exists). By \eqref{K-ext Eq Lem character-weight different 1} we conclude that $\nu_{j_0}=\lambda_{j_0}+\alpha-(p-1)\alpha_{13}$.
\end{proof}

\begin{lemma}\label{K-ext Lem Yj0 isomorphism}
    Let $\alpha\in\Phi^+\cup\set{0}$, $j_0\in\cJ$ and $\lambda\in C(\alpha,j_0)$. Let $\nu=\lambda+\alpha_{j_0}-(p-1)\un{\alpha_{13}}$. Then $X_{13,j_0}$ induces an isomorphism $L(\lambda)_{\nu}\to L(\lambda)_{\nu+\alpha_{13,j_0}}$ unless $\alpha=0$, $\lambda_{j_0,1}-\lambda_{j_0,3}=p-2$ and $\lambda_{j,1}-\lambda_{j,3}\geq p-1~\forall\,j\neq j_0$.
\end{lemma}

\begin{proof}
    If $\lambda_{j,1}-\lambda_{j,3}\leq p-2$ for some $j\neq j_0$, then we have $L(\lambda_{j})_{\nu_j}=0$, hence $L(\lambda)_{\nu}=L(\lambda)_{\nu+\alpha_{13,j_0}}=0$. Otherwise, by Lemma \ref{K-ext Lem Y maps into}, we have $X_{13,j_0}=-X_{13}^{(j_0),\op{alg}}$. Hence it suffices to show that the map
    \begin{equation*}
        X_{13}^{\op{alg}}:L(\lambda_{j_0})_{\nu_{j_0}}\to L(\lambda_{j_0})_{\nu_{j_0}+\alpha_{13}}
    \end{equation*}
    is an isomorphism, which follows from Lemma \ref{K-ext Lem Y isomorphism}. 
\end{proof}

\begin{lemma}\label{K-ext Lem Yj0 injective}
    Let $\alpha\in\Phi^+\cup\set{0}$ and $j_0\in\cJ$. Let $\lambda\in C(\alpha,j_0)$ if $\alpha\neq0$ and $\lambda\in X_1(\un{T})$ with $\lambda_{j_0,1}-\lambda_{j_0,3}\leq 2p-3$ if $\alpha=0$. Given a set of weights $\nu\in X(\un{T})$ each satisfying $\nu_{j_0}=\lambda_{j_0}+\alpha-(p-1)\alpha_{13}$. Then the map $X_{13,j_0}:\oplus_{\nu}L(\lambda)_{\nu}\to L(\lambda)$ is injective.
\end{lemma}

\begin{proof}
    By Lemma \ref{K-ext Lem Xalpha injective}(ii), it suffices to show that each $X_{13}^{\op{alg}}:L(\lambda_{j_0})_{\nu_{j_0}}\to L(\lambda_{j_0})_{\nu_{j_0}+\alpha_{13}}$ is injective. Indeed, if $\alpha=\alpha_{13}$ and $\lambda_{j_0}|_{\SL_3}=(p-2,p-2)$, then the injectivity follows from Lemma \ref{K-ext Lem Y isomorphism}. In the rest of the cases, we have $\nu_{j_0}+i\alpha_{13}\nleq\lambda_{j_0}$ for $i\geq p$, and
    \begin{equation*}
        \ang{\nu_{j_0},\alpha_{13}^{\vee}}=\lambda_{j_0,1}-\lambda_{j_0,3}+\ang{\alpha,\alpha_{13}^{\vee}}-(2p-2)\leq-1.
    \end{equation*}
    Then the injectivity follows from Lemma \ref{K-ext Lem Xalg injective}.
\end{proof}

\section{The \texorpdfstring{$K$}{.}-extensions between Serre weights}\label{K-ext sec argument}

We keep the notation of the previous sections and introduce more notation.

We write $K\eqdef\GL_3(\OL)$. Let $I\eqdef\set{A\in K:(A\,\mod\,p)\in B(\Fq)}$ be the Iwahori subgroup. Let $I_1\eqdef\set{A\in K:(A\,\mod\,p)\in U(\Fq)}$ be the pro-$p$ Iwahori subgroup. Let $K_1\eqdef\op{I}_3+p\Mat_3(\OL)$ be the first congruence subgroup. Let $T_1\eqdef T(1+p\OL)$. Let $Z_1\cong1+p\OL$ be the center of $K_1$. Let $K_2\eqdef\op{I}_3+p^2\Mat_3(\OL)=K_1^p$. In particular, for $\lambda\in X_1(\un{T})$ we have $F(\lambda)^{I_1}=L(\lambda)_{\lambda}$, which is 1-dimensional over $\FF$.

\begin{theorem}\label{K-ext Thm Calpha}
    Let $\alpha\in\Phi^+\cup\set{0}$, $j_0\in\cJ$ and $\lambda\in C(\alpha,j_0)$ (see Definition \ref{K-ext Def Calpha} and Definition \ref{K-ext Def Calphaj0}). We write $\sigma=F(\lambda)$ and $\tau=F(\lambda+\alpha_{j_0})$. Then we have
    \begin{equation*}
        \Ext^1_{K/Z_1}(\tau,\sigma)\cong\Ext^1_{\GL_3(\Fq)}(\tau,\sigma).
    \end{equation*}
\end{theorem} 

\proof

Let $V$ be a $K/Z_1$-representation that is a non-split extension of $\tau$ by $\sigma$
\begin{equation*}
    0\to\sigma\to V\to\tau\to0.
\end{equation*}
The goal is to show that $V$ is $K_1$-invariant.
We let $0\neq v\in V$ be an $H$-eigenvector whose image in $\tau$ is a nonzero element of $\tau^{I_1}$. In particular, we have $v\in V^{\lambda+\alpha_{j_0}}$. We start with the $U(\OL)$-action on $v$.

\begin{lemma}\label{K-ext Lem U invariant}
    We may choose $v$ such that $X_{13,j}{v}=X_{12,j}{v}=X_{23,j}{v}=0$ for all $j\in\cJ$. Equivalently, $v$ is fixed by $U(\OL)$.
\end{lemma}

\begin{proof}
    We give a proof in the case $f\geq2$. The case $f=1$ being similar and simpler. In the case $\alpha=0$, we modify $j_0$ such that either $\lambda_{j_0,1}-\lambda_{j_0,3}\leq p-2$ or $\lambda_{j,1}-\lambda_{j,3}\leq p-2$ for some $j\neq j_0$.

    First we prove that it is possible to choose $v$ such that $X_{13,j_0}v=0$. We start with an arbitrary choice of $v$. Since $v$ maps into $\tau^{I_1}$ in $\tau$, the image of $X_{13,j_0}v$ in $\tau$ is zero, hence $X_{13,j_0}v\in\sigma$. By Lemma \ref{K-ext Y changes characters} we have $X_{13,j_0}v\in\sigma^{\lambda+\alpha_{j_0}+\alpha_{13,j_0}}$, which equals $L(\lambda)_{\lambda+\alpha_{j_0}+\alpha_{13,j_0}-(p-1)\un{\alpha_{13}}}$ by Lemma \ref{K-ext Lem character-weight same}. Then by Lemma \ref{K-ext Lem Yj0 isomorphism} and the choice of $j_0$, there exists $w\in L(\lambda)_{\lambda+\alpha_{j_0}-(p-1)\un{\alpha_{13}}}\subseteq\sigma^{\lambda+\alpha_{j_0}}$ such that $X_{13,j_0}w=X_{13,j_0}v$. We conclude by replacing $v$ with $v-w$.

    Next we prove that for $v$ as in the previous paragraph, we have $X_{\beta,j_1}v=0$ for all $j_1\neq j_0$ and $\beta\in\Phi^+$. Similar to the previous paragraph, we have $X_{\beta,j_1}v\in\sigma^{\lambda+\alpha_{j_0}+\beta_{j_1}}$. By Lemma \ref{K-ext Lem character-weight different} if $\alpha\neq0$ and by Lemma \ref{K-ext Lem character-weight same} (replacing $j_0$ with $j_1$) if $\alpha=0$, $\sigma^{\lambda+\alpha_{j_0}+\beta_{j_1}}$ is a direct sum of weight spaces $L(\lambda)_{\nu}$ with $\nu_{j_0}=\lambda_{j_0}+\alpha-(p-1)\alpha_{13}$. Since $X_{13,j_0}\bra{X_{\beta,j_1}v}=X_{\beta,{j_1}}\bra{X_{13,j_0}v}=0$ by the choice of $v$, we conclude from Lemma \ref{K-ext Lem Yj0 injective} that $X_{\beta,j_1}v=0$.

    Finally, we prove that $X_{12,j_0}v=X_{13,j_0}v=0$ for $v$ as above. Let $\beta\in\set{\alpha_{12},\alpha_{23}}$. Then as above, we have $X_{\beta,j_0}v\in\sigma^{\lambda+\alpha_{j_0}+\beta_{j_0}}=L(\lambda)_{\lambda+\alpha_{j_0}+\beta_{j_0}-(p-1)\un{\alpha_{13}}}$ by Lemma \ref{K-ext Lem character-weight same}. Since $X_{13,j_0-1}\bra{X_{\beta,j_0}v}=X_{\beta,j_0}\bra{X_{13,j_0-1}v}=0$ by the previous paragraph, and $\lambda\in C(\alpha,j_0)$ implies $\lambda_{j_0-1,1}-\lambda_{j_0-1,3}\leq 2p-3$, we conclude from Lemma \ref{K-ext Lem Yj0 injective} (replacing $(\alpha,j_0)$ with $(0,j_0-1)$) that $X_{\beta,j_0}v=0$.
\end{proof}

From now on, we fix a choice of $v$ as in Lemma \ref{K-ext Lem U invariant}. In particular, $v$ is fixed by $u_{12}(t)$, $u_{13}(t)$ and $u_{23}(t)$ for all $t\in\OL$. Then we turn to the $T_1$-action on $v$. We fix $0\neq u\in\sigma^{I_1}=L(\lambda)_{\lambda}$, which is unique up to scalar.

\begin{lemma}\label{K-ext Lem Torus-action}
    There exist $k_{1,j},k_{2,j},k_{3,j}\in\FF$ with $k_{1,j}+k_{2,j}+k_{3,j}=0$ for all $j$ such that 
    \begin{equation*}
        \op{diag}\bigbra{1+pa,1+pb,1+pc}v=v+\ssum_{j=0}^{f-1}\bigbra{k_{1,j}\ovl{a}^{p^j}+k_{2,j}\ovl{b}^{p^j}+k_{3,j}\ovl{c}^{p^j}}u\quad\forall\,a,b,c\in\OL.
    \end{equation*}
    If moreover $\alpha\neq0$, then we have $k_{1,j}=k_{2,j}=k_{3,j}=0$ for all $j$.
\end{lemma}

\begin{proof}
    For $g\in T_1$ we write $w_g\eqdef g(v)-v\in V$. Since $\tau$ is $K_1$-invariant, the image of $w_g$ in $\tau$ is zero, hence $w_g\in\sigma$ and is fixed by $K_1$. Since $v$ is fixed by $U(\OL)$ and since $T_1$ normalizes $U(\OL)$, we deduce that $w_g$ is also fixed by $U(\OL)$. Hence we have $w_g\in\sigma^{I_1}=\FF u$. Therefore, the map $g\mapsto w_g$ defines an element of 
    \begin{equation*}
        H^1(T_1/Z_1,\FF u)\cong\Hom(T_1/Z_1,\FF)\cong\Hom(\OL^3/\OL,\FF)\cong\Hom(\Fq^3/\Fq,\FF),
    \end{equation*}
    which is isomorphic to $\Ker\!\bigbra{\FF^3\to\FF:(a,b,c)\mapsto a+b+c}^{\oplus f}$. This proves the first statement.
    
    For $g\in T_1$, since $g$ commutes with $H$ and $v\in V^{\lambda+\alpha_{j_0}}$, we deduce that $w_g\in\sigma^{\lambda+\alpha_{j_0}}$. The previous paragraph also shows that $w_g\in\sigma^{I_1}\subseteq\sigma^{\lambda}$. Hence, in the case $\alpha\neq0$ we have $w_g\in\sigma^{\lambda+\alpha_{j_0}}\cap\sigma^{\lambda}=0$ for all $g\in T_1$, which proves the second statement.
\end{proof}

Next we turn to the $I$-action on $v$. For $\beta\in\Phi^-$ and $j\in\cJ$, we write $w_{\beta,j}\eqdef X_{\beta,j-1}^pv\in\sigma$ (since $v$ has image $\tau^{I_1}$ in $\tau$). By Lemma \ref{K-ext Y changes characters} we have $w_{\beta}\in\sigma^{\lambda+\alpha_{j_0}+p\beta_{j-1}}=\sigma^{\lambda+\alpha_{j_0}+\beta_{j}}$. 

\begin{lemma}\label{K-ext Lem I-structure}
    Let $j\in\cJ$.
\begin{enumerate}
    \item 
    If $k_{1,j}=k_{2,j}$, then $w_{21,j}=0$ unless $\alpha=\alpha_{12}$ and $j=j_0$. If $k_{1,j}\neq k_{2,j}$, then $w_{21,j}\neq0$.
    \item 
    If $k_{2,j}=k_{3,j}$, then $w_{32,j}=0$ unless $\alpha=\alpha_{23}$ and $j=j_0$. If $k_{2,j}\neq k_{3,j}$, then $w_{32,j}\neq0$.
    \item 
    If $w_{21,j}=w_{32,j}=0$, then $w_{31,j}=0$ unless $\alpha=\alpha_{13}$ and $j=j_0$.
\end{enumerate}
\end{lemma}

\begin{proof}
(i). Using the matrix identities (for $s,t\in\Fq\x$)
\begin{equation*}
\begin{aligned}
    u_{12}([s])u_{21}(p[t])&=u_{21}\bigbra{p[t](1\!+\!p[st])\inv}\op{diag}\bigbra{1\!+\!p[st],(1\!+\!p[st])\inv,1}u_{12}\bigbra{[s](1\!+\!p[st])\inv};\\
    u_{23}([s])u_{21}(p[t])&=u_{21}(p[t])u_{23}([s]);\\
    u_{13}([s])u_{21}(p[t])&=u_{21}(p[t])u_{23}(-p[st])u_{13}([s]),
\end{aligned}
\end{equation*}
we compute that (for $s\in\Fq\x$)
\begin{equation*}
\begin{aligned}
    u_{12}([s])w_{21,j}&=\ssum_{t\in\Fq\x}t^{-p^j}u_{12}([s])u_{21}(p[t])v\\
    &\hspace{-3em}=\ssum_{t\in\Fq\x}t^{-p^j}u_{21}\bigbra{p[t](1\!+\!p[st])\inv}\op{diag}\bigbra{1\!+\!p[st],(1\!+\!p[st])\inv,1}u_{12}\bigbra{[s](1\!+\!p[st])\inv}v\\
    &\hspace{-3em}=\ssum_{t\in\Fq\x}t^{-p^j}u_{21}\bigbra{p[t]}\op{diag}\bigbra{1\!+\!p[st],(1\!+\!p[st])\inv,1}v\quad\text{(since $V$ is invariant under $K_2$)}\\
    &\hspace{-3em}=\ssum_{t\in\Fq\x}t^{-p^j}u_{21}\bigbra{p[t]}\bbbra{v+\ssum_{j_1=0}^{f-1}\bigbra{k_{1,j_1}\bra{st}^{p^{j_1}}-k_{2,j_1}\bra{st}^{p^{j_1}}}u}\\
    &\hspace{-3em}=w_{21,j}+\ssum_{j_1=0}^{f-1}\bra{k_{1,j_1}-k_{2,j_1}}s^{p^{j_1}}\bbra{\ssum_{t\in\Fq\x}t^{p^{j_1}-p^j}}u=w_{21,j}-\bra{k_{1,j}-k_{2,j}}s^{p^j}u;\\
\end{aligned}
\end{equation*}
\begin{equation*}
\begin{aligned}
    u_{23}([s])w_{21,j}&=\ssum_{t\in\Fq\x}t^{-p^j}u_{23}([s])u_{21}(p[t])v=\ssum_{t\in\Fq\x}t^{-p^j}u_{21}(p[t])u_{23}([s])v\\
    &=\ssum_{t\in\Fq\x}t^{-p^j}u_{21}(p[t])v=w_{21,j};\\
    u_{13}([s])w_{21,j}&=\ssum_{t\in\Fq\x}t^{-p^j}u_{13}([s])u_{21}(p[t])v=\ssum_{t\in\Fq\x}t^{-p^j}u_{21}(p[t])u_{23}(-p[st])u_{13}([s])v\\
    &=\ssum_{t\in\Fq\x}t^{-p^j}u_{21}(p[t])v=w_{21,j}.
\end{aligned}
\end{equation*}
Hence we have (for $j_1\in\cJ$)
\begin{equation}\label{K-ext Eq I-structure on w21-2}
\begin{aligned}
    X_{12,j_1}w_{21,j}&=\ssum_{s\in\Fq\x}s^{-p^{j_1}}u_{12}([s])w_{21,j}\\
    &=\ssum_{s\in\Fq\x}s^{-p^{j_1}}\bigbra{w_{21,j}-\bra{k_{1,j}-k_{2,j}}s^{p^j}u}=\delta_{j=j_1}\bra{k_{1,j}-k_{2,j}};\\
    X_{23,j_1}w_{21,j}&=\ssum_{s\in\Fq\x}s^{-p^{j_1}}u_{23}([s])w_{21,j}=\ssum_{s\in\Fq\x}s^{-p^{j_1}}w_{21,j}=0;\\
    X_{13,j_1}w_{21,j}&=\ssum_{s\in\Fq\x}s^{-p^{j_1}}u_{13}([s])w_{21,j}=\ssum_{s\in\Fq\x}s^{-p^{j_1}}w_{21,j}=0.
\end{aligned}
\end{equation}
Suppose that $k_{1,j}=k_{2,j}$. Since $w_{21,j}\in\sigma$ is invariant under $K_1$, we deduce from \eqref{K-ext Eq I-structure on w21-2} that $w_{21,j}\in\sigma^{I_1}\cap\sigma^{\lambda+\alpha_{j_0}+\alpha_{21,j}}\subseteq\sigma^{\lambda}\cap\sigma^{\lambda+\alpha_{j_0}+\alpha_{21,j}}$, which equals zero unless $\alpha=\alpha_{12}$ and $j=j_0$.

\vspace{1em}

(ii). Similar to (i), using the matrix identities (for $s,t\in\Fq\x$)
\begin{equation*}
\begin{aligned}
    u_{12}([s])u_{32}(p[t])&=u_{32}(p[t])u_{12}([s]);\\
    u_{23}([s])u_{32}(p[t])&=u_{32}\bigbra{p[t](1\!+\!p[st])\inv}\op{diag}\bigbra{1,1\!+\!p[st],(1\!+\!p[st])\inv}u_{23}\bigbra{[s](1\!+\!p[st])\inv};\\
    u_{13}([s])u_{32}(p[t])&=u_{32}(p[t])u_{12}(p[st])u_{13}([s]),
\end{aligned}
\end{equation*}
we compute that (for $s\in\Fq\x$ and $j_1\in\cJ$)
\begin{equation}\label{K-ext Eq I-structure on w32}
\begin{split}
    u_{12}([s])w_{32,j}&=w_{32,j};\\
    u_{23}([s])w_{32,j}&=w_{32,j}-(k_{2,j}-k_{3,j})s^{p^j}u;\\
    u_{13}([s])w_{32,j}&=w_{32,j};
\end{split}
\hspace{3em}
\begin{split}
    X_{12,j_1}w_{32,j}&=0;\\
    X_{23,j_1}w_{32,j}&=\delta_{j=j_1}(k_{2,j}-k_{3,j})u;\\
    X_{13,j_1}w_{32,j}&=0.
\end{split}
\end{equation}
In particular, in the case $k_{2,j}=k_{3,j}$, by \eqref{K-ext Eq I-structure on w32} we have $w_{32,j}\in\sigma^{I_1}\cap\sigma^{\lambda+\alpha_{j_0}+\alpha_{32,j}}\subseteq\sigma^{\lambda}\cap\sigma^{\lambda+\alpha_{j_0}+\alpha_{32,j}}$, which equals zero unless $\alpha=\alpha_{23}$ and $j=j_0$.

\vspace{1em}

(iii). By Lemma \ref{K-ext Lem Iwasawa} and raising to $p$-th power, we have $u_{32}(p[t])v=v-\sum_{j=0}^{f-1}t^{p^j}w_{32,j}$ for all $t\in\Fq\x$. Then similar to (i), using the matrix identities (for $s,t\in\Fq\x$)
\begin{equation*}
\begin{aligned}
    u_{12}([s])u_{31}(p[t])&=u_{31}(p[t])u_{32}(-p[st])u_{12}([s]);\\
    u_{23}([s])u_{31}(p[t])&=u_{31}(p[t])u_{21}(p[st])u_{23}([s]);\\
    u_{13}([s])u_{31}(p[t])&=u_{31}\bigbra{p[t](1\!+\!p[st])\inv}\op{diag}\bigbra{1\!+\!p[st],1,(1\!+\!p[st])\inv}u_{13}\bigbra{[s](1\!+\!p[st])\inv},
\end{aligned}
\end{equation*}
we compute that (for $s\in\Fq\x$ and $j_1\in\cJ$)
\begin{equation}\label{K-ext Eq I-structure on w31}
\begin{split}
    u_{12}([s])w_{31,j}&=w_{31,j}-s^{p^j}w_{32};\\
    u_{23}([s])w_{31,j}&=w_{31,j}+s^{p^j}w_{21};\\
    u_{13}([s])w_{31,j}&=w_{31,j}-(k_{1,j}-k_{3,j})s^{p^j}u;
\end{split}
\hspace{3em}
\begin{split}
    X_{12,j_1}w_{31,j}&=\delta_{j=j_1}w_{32,j};\\
    X_{23,j_1}w_{31,j}&=-\delta_{j=j_1}w_{21,j};\\
    X_{13,j_1}w_{31,j}&=\delta_{j=j_1}(k_{1,j}-k_{3,j})u.
\end{split}
\end{equation}
Suppose that $w_{21,j}=w_{32,j}=0$. We deduce from \eqref{K-ext Eq I-structure on w21-2} and \eqref{K-ext Eq I-structure on w32} that $k_{1,j}=k_{2,j}=k_{3,j}$. Hence by \eqref{K-ext Eq I-structure on w31} have $w_{31,j}\in\sigma^{I_1}\cap\sigma^{\lambda+\alpha_{j_0}+\alpha_{31,j}}\subseteq\sigma^{\lambda}\cap\sigma^{\lambda+\alpha_{j_0}+\alpha_{31,j}}$, which equals zero unless $\alpha=\alpha_{13}$ and $j=j_0$. 
\end{proof}

\begin{lemma}\label{K-ext Lem w31=0}
    If $w_{31,j}=0$ for all $j$, then $V$ is invariant under $K_1$.
\end{lemma}

\begin{proof}
    In this case, $v$ is invariant under $U(\OL)$ and $U_{31}(p\OL)$. Since they generate $I_1$, we deduce that $v\in V^{I_1}$. Then by Frobenius reciprocity we have a $K$-equivariant map $\Ind_{I}^K\FF v\to V$, which is surjective since $v$ has nonzero image in $\tau$ that is the $K$-cosocle of $V$. Since $\Ind_{I}^K\FF v$ is $K_1$-invariant, we conclude that $V$ is also $K_1$-invariant.
\end{proof}

Now we are ready to prove Theorem \ref{K-ext Thm Calpha}. If $w_{31,j}=0$ for all $j$, then by Lemma \ref{K-ext Lem w31=0} we know that $V$ is invariant under $K_1$, as desired. From now on, we assume that $w_{31,j}\neq0$ for some $j$. If $\alpha\neq0$, then by Lemma \ref{K-ext Lem Torus-action} and Lemma \ref{K-ext Lem I-structure} we have $w_{31,j}=0$ for all $j\neq j_0$, hence we must have $w_{31,j_0}\neq0$. If $\alpha=0$, then we modify $j_0$ such that $w_{31,j_0}\neq0$. We are going to deduce a contradiction.

\vspace{1em}

We let $W\subseteq V$ be the $I$-subrepresentation generated by $v$. %By Lemma \ref{K-ext Lem Torus-action} and Lemma \ref{K-ext Lem I-structure}(i),(ii),(iii) (and its proof), we have the following possibilities:
%\begin{equation*}
    %W=
%\begin{cases}
    %\FF v\oplus\FF w_{31}\oplus\FF w_{32}\oplus\FF u&\text{if}~\alpha=0,~k_1=k_2\neq k_3\\
    %\FF v\oplus\FF w_{31}\oplus\FF w_{21}\oplus\FF u&\text{if}~\alpha=0,~k_1\neq k_2=k_3\\
    %\FF v\oplus\FF w_{31}\oplus\bigbra{\FF w_{21}\oplus\FF w_{32}}\oplus\FF u&\text{if}~\alpha=0,~k_1\neq k_2~\text{and}~k_2\neq k_3;
%\end{cases}
%\end{equation*}
%\begin{equation*}
    %W=
%\begin{cases}
    %\FF v\oplus\FF w_{31}=\FF v\oplus\FF u&\text{if}~\alpha=\alpha_{13}\\
    %\FF v\oplus\FF w_{31}\oplus\FF w_{21}=\FF v\oplus\FF w_{31}\oplus\FF u&\text{if}~\alpha=\alpha_{12}\\
    %\FF v\oplus\FF w_{31}\oplus\FF w_{32}=\FF v\oplus\FF w_{31}\oplus\FF u&\text{if}~\alpha=\alpha_{23}.
%\end{cases}
%\end{equation*}
Then by Frobenius reciprocity, the $I$-equivariant inclusion $W\into V$ induces a $K$-equivariant map
\begin{equation*}
    f:\Ind_I^{K}W\to V,
\end{equation*}
which is surjective since $v\in\op{Im}(f)$ has nonzero image in $\tau$ that is the $K$-cosocle of $V$. For $w\in W$, we write $[1,w]\in\Ind_I^K W$ the element supported on $I$ whose value at the identity matrix is given by $w$. We separate the cases $f\geq2$ and $f=1$.
\vspace{1em}

\noindent\textbf{Case 1.} $f\geq2$.

Since $\lambda\in C(\alpha,j_0)$ and $w_{31,j_0}\neq0$, by Lemma \ref{K-ext Lem I-structure} at least one of the following holds:
\begin{equation*}
\begin{split}
    \text{(i).}~&\alpha=\alpha_{13}~\text{and}~\lambda_{j_0-1,1}-\lambda_{j_0-1,2}\leq p-2;\\
    \text{(ii).}~&\alpha=\alpha_{12};\\
    \text{(iii).}~&\alpha=0~\text{and}~k_{1,j_0}\neq k_{2,j_0};
\end{split}
\hspace{2em}
\begin{split}
    \text{(iv).}~&\alpha=\alpha_{13}~\text{and}~\lambda_{j_0-1,2}-\lambda_{j_0-1,3}\leq p-2;\\
    \text{(v).}~&\alpha=\alpha_{23};\\
    \text{(vi).}~&\alpha=0~\text{and}~k_{2,j_0}\neq k_{3.j_0}.
\end{split}
\end{equation*}
From now on, we concentrate on the cases (i),(ii),(iii), the other cases being symmetric. We consider the element
%\begin{equation*}
    %\text{(i).~$\alpha=\alpha_{13}$ and $\lambda_1-\lambda_2\leq p-4$;\quad(ii).~$\alpha=\alpha_{12}$;\quad(iii).~$\alpha=0$ and $k_1\neq k_2$.}
%\end{equation*}
%\begin{equation*}
    %y\eqdef\bbra{\sprod_{j\neq j_0-1}X_{j}}X_{13,j_0-1}^{p-1}X_{12,j_0-1}^{m}\smat{&&1\\&1&\\1&&}[1,v]\in\Ind_I^K W,
%\end{equation*}
\begin{equation*}
    y\eqdef\bbra{\sprod_{j=0}^{f-1}X_{13,j}^{m_j}X_{12,j}^{n_j}}\smat{&&1\\&1&\\1&&}[1,v]\in\Ind_I^K W
\end{equation*}
with
%\begin{equation*}
%\begin{aligned}
    %m_j&=
%\begin{cases}
    %\lambda_{j,1}-\lambda_{j,2}&\text{if}~j\neq j_0,j_0-1\\
    %\lambda_{j,1}-\lambda_{j,2}-1+\alpha_1&\text{if}~j=j_0\\
    %p-1&\text{if}~j=j_0-1;
%\end{cases}\\
    %n_j&=
%\begin{cases}
    %\lambda_{j,2}-\lambda_{j,3}&\text{if}~j\neq j_0,j_0-1\\
    %\lambda_{j,2}-\lambda_{j,3}+\alpha_2&\text{if}~j=j_0\\
    %\max\set{0,\lambda_{j_0-1,1}-\lambda_{j_0-1,3}-(p-2)}&\text{if}~j=j_0-1.
%\end{cases}
%\end{aligned}
%\end{equation*}
\begin{equation}\label{K-ext Eq mjnj}
    (m_j,n_j)=
\begin{cases}
    \bigbra{\lambda_{j,1}-\lambda_{j,2},\lambda_{j,2}-\lambda_{j,3}}&\text{if}~j\neq j_0,j_0-1\\
    \bigbra{\lambda_{j_0,1}-\lambda_{j_0,2}-1+\alpha_1,\lambda_{j_0,2}-\lambda_{j_0,3}+\alpha_2}&\text{if}~j=j_0\\
    \bigbra{p-1,\max\set{0,\lambda_{j_0-1,1}-\lambda_{j_0-1,3}-(p-2)}}&\text{if}~j=j_0-1.
\end{cases}
\end{equation}
%\begin{equation*}
%\begin{aligned}
    %X_j&\eqdef
%\begin{cases}
    %X_{13,j}^{\lambda_{j,1}-\lambda_{j,2}}X_{12,j}^{\lambda_{j,2}-\lambda_{j,3}}&\text{if}~j\neq j_0,j_0-1\\
    %X_{13,j}^{\lambda_{j,1}-\lambda_{j,2}-1+\alpha_1}X_{12,j}^{\lambda_{j,2}-\lambda_{j,3}+\alpha_2}&\text{if}~j=j_0;
%\end{cases}\\
    %m&\eqdef\max\set{0,\lambda_{j_0-1,1}-\lambda_{j_0-1,3}-(p-2)}.
%\end{aligned}
%\end{equation*}
%We also write $\ovl{X}_j\eqdef\smat{&&1\\&1&\\1&&}X_j\smat{&&1\\&1&\\1&&}$ for all $j\neq j_0$.

\begin{lemma}\label{K-ext Lem image zero}
    The element $f(y)\in V$ has image zero in $\tau$.
\end{lemma}

\begin{proof}
    By definition, we have 
    \begin{equation*}
        f(y)=\bbra{\sprod_{j=0}^{f-1}X_{13,j}^{m_j}X_{12,j}^{n_j}}\smat{&&1\\&1&\\1&&}v=\smat{&&1\\&1&\\1&&}\bbra{\sprod_{j=0}^{f-1}X_{31,j}^{m_j}X_{32,j}^{n_j}}v.
    \end{equation*}
    %\begin{equation*}
    %\begin{aligned}
        %f(y)&=\bbra{\sprod_{j\neq j_0-1}X_j}X_{13,j_0-1}^{p-1}X_{12,j_0-1}^{m}\smat{&&1\\&1&\\1&&}v\\
        %&=\smat{&&1\\&1&\\1&&}(???)X_{31,j_0-1}^{p-1}X_{32,j_0-1}^{m}v\in V.
    %\end{aligned}
    %\end{equation*}
    Since $v$ maps into $\tau^{I_1}$ in $\tau$ and since
    \begin{equation*}
    \begin{aligned}
        &m_j+n_j=\lambda_{j,1}-\lambda_{j,3}\quad\forall\,j\neq j_0,j_0-1;\\
        &m_{j_0}\!+\!n_{j_0}=\bra{\lambda_{j_0,1}\!-\!\lambda_{j_0,2}\!-\!1\!+\!\alpha_1}+\bra{\lambda_{j_0,2}\!-\!\lambda_{j_0,3}\!+\!\alpha_2}\geq\bra{\lambda_{j_0,1}\!+\!\alpha_1}-\bra{\lambda_{j_0,3}\!-\!\alpha_3}-(p\!-\!1);\\
        &m_{j_0-1}\!+\!n_{j_0-1}\geq(p\!-\!1)+\bra{\lambda_{j_0-1,1}\!-\!\lambda_{j_0-1,3}\!-\!(p\!-\!2)}=\lambda_{j_0-1,1}-\lambda_{j_0-1,3}+1,
    \end{aligned}
    \end{equation*}
    we conclude by applying Lemma \ref{K-ext Lem operator mapsto zero} to $\tau$.
\end{proof}

By Lemma \ref{K-ext Lem image zero}, we have $f(y)\in\sigma$. Then by Lemma \ref{K-ext Y changes characters} we have $f(y)\in\sigma^{\mu}$ with
\begin{equation*}
    \mu\eqdef w_0\bra{\lambda+\alpha_{j_0}}+\ssum_{j=0}^{f-1}\bigbra{m_j\alpha_{13,j}+n_j\alpha_{12,j}}.
\end{equation*}

\begin{lemma}\label{K-ext Lem f(y) first weight}
    We have $f(y)\in\oplus_{\nu}L(\lambda)_{\nu}$ with each $\nu$ satisfying $\nu_{j,2}\neq\mu_{j,2}$ for some $j$.
\end{lemma}

\begin{proof}
    By definition, we have
    \begin{equation*}
        \mu_j=
    \begin{cases}
        \bigbra{\lambda_{j,1},\lambda_{j,3},\lambda_{j,2}}&\text{if}~j\neq j_0,j_0-1\\
        \bigbra{\lambda_{j_0,1}-1,\lambda_{j_0,3},\lambda_{j_0,2}+1}&\text{if}~j=j_0\\
        \bigbra{\lambda_{j_0-1,3}+(p-1)+n_{j_0-1},\lambda_{j_0-1,2}-n_{j_0-1},\lambda_{j_0-1,1}-(p-1)}&\text{if}~j=j_0-1.
    \end{cases}
    \end{equation*}
    Let $\nu\in X(\un{T})$ such that $0\neq L(\lambda)_{\nu}\subseteq\sigma^{\mu}$. Then as in \eqref{K-ext Eq Lem lambda-alpha12 1} there exist $\un{c}_1,\un{c}_2\in\ZZ^f$ such that 
    \begin{equation*}
        \nu=\mu-\ssum_{j=0}^{f-1}\bigbra{\bra{pc_{1,j}-c_{1,j-1}}\alpha_{12,j}+\bra{pc_{2,j}-c_{2,j-1}}\alpha_{23,j}}.
    \end{equation*}
    Since $\lambda_{j,3}\leq\nu_{j,1},\nu_{j,3}\leq\lambda_{j,1}$ for all $j$, as in the proof of Lemma \ref{K-ext Lem lambda-alpha12} we deduce that $-1\leq c_{1,j}\leq0$ and $0\leq c_{2,j}\leq1$ for all $j$. 
    
    Suppose on the contrary that $\nu_{j,2}=\mu_{j,2}$ for all $j$, then we have $c_{1,j}=c_{2,j}$ for all $j$, hence $c_{1,j}=c_{2,j}=0$ for all $j$ and $\nu=\mu$. Since $\mu_{j_0-1,1}=\lambda_{j_0-1,3}+(p-1)+n_{j_0-1}\geq\lambda_{j_0-1,1}+1$, we have $L(\lambda)_{\mu}=0$, which is a contradiction.
\end{proof}

Next, we consider the element (note that these operators commute with each other)
\begin{equation*}
    y'\eqdef X_{13,j_0}y=\bbra{\sprod_{j\neq j_0-1}X_{13,j}^{m_j}X_{12,j}^{n_j}}X_{12,j_0-1}^{n_{j_0-1}}\smat{&&1\\&1&\\1&&}[1,w_{31,j_0}]\in\Ind_I^K W.
\end{equation*}
Since $w_{31,j_0}\in\sigma^{\lambda+\alpha_{j_0}-\alpha_{13,j_0}}$, by Lemma \ref{K-ext Y changes characters} we have $f(y')\in\sigma^{\mu'}$ with
\begin{equation*}
    \mu'\eqdef w_0\bra{\lambda+\alpha_{j_0}-\alpha_{13,j_0}}+\ssum_{j\neq j_0-1}\bigbra{m_j\alpha_{13,j}+n_j\alpha_{12,j}}+n_{j_0-1}\alpha_{12,j_0-1}.
\end{equation*}
In particular, we have $\mu'_{j,2}=\mu_{j,2}$ for all $j$.

\begin{lemma}\label{K-ext Lem f(y) second weight}
    The element $f(y')\in\sigma$ has nonzero image in $L(\lambda)_{\mu'}$.
\end{lemma}

\begin{proof}
    By definition, we have
    \begin{equation*}
    \begin{aligned}
        f(y')&=\bbra{\sprod_{j\neq j_0-1}X_{13,j}^{m_j}X_{12,j}^{n_j}}X_{12,j_0-1}^{n_{j_0-1}}\smat{&&1\\&1&\\1&&}w_{31,j_0}\\
        &=\smat{&&1\\&1&\\1&&}\bbra{\sprod_{j\neq j_0-1}X_{31,j}^{m_j}X_{32,j}^{n_j}}X_{32,j_0-1}^{n_{j_0-1}}w_{31,j_0}\in\sigma\subseteq V.
    \end{aligned}
    \end{equation*}
    By \eqref{K-ext Eq I-structure on w21-2}, \eqref{K-ext Eq I-structure on w32} and \eqref{K-ext Eq I-structure on w31} we have $w_{31,j_0}\in\sigma^{I_1}$ (resp.\;$X_{23,j_0}w_{31,j_0}=X_{13,j_0-1}w_{31,j_0}=0$, $X_{12,j_0-1}w_{31,j_0}=0$) if $\alpha=\alpha_{13}$ (resp.\;$\alpha=\alpha_{12}$, $\alpha=0$). Hence by Lemma \ref{K-ext Lem lambda-alpha12} and Lemma \ref{K-ext Lem lambda-alpha13} we have
    \begin{equation*}
        0\neq w_{31,j_0}\in L(\lambda)_{\lambda+\alpha_{j_0}-\alpha_{13,j_0}}=L(\lambda_0)_{\lambda_0}\otimes\cdots\otimes L(\lambda_{j_0})_{\lambda_{j_0}+\alpha-\alpha_{13}}\otimes\cdots\otimes L(\lambda_{f-1})_{\lambda_{f-1}}.
    \end{equation*}
    Since each $L(\lambda_j)_{\lambda_j}$ is $1$-dimensional, we can write $w_{31,j_0}=v_{\lambda_0}\otimes\cdots\otimes w\otimes\cdots\otimes v_{\lambda_{f-1}}$ with $w\in L(\lambda_{j_0})_{\lambda_{j_0}+\alpha-\alpha_{13}}$. We claim that (see \eqref{K-ext Eq mjnj} for $m_j$ and $n_j$)
    \begin{enumerate}
        \item 
        $\bigbra{X_{31}^{\op{alg}}}^{m_j}\bigbra{X_{32}^{\op{alg}}}^{n_j}v_{\lambda_{j}}\neq0$ for $j\neq j_0,j_0-1$;
        \item 
        $\bigbra{X_{32}^{\op{alg}}}^{n_{j_0-1}}v_{\lambda_{j_0-1}}\neq0$;
        \item 
        $\bigbra{X_{31}^{\op{alg}}}^{m_{j_0}}\bigbra{X_{32}^{\op{alg}}}^{n_{j_0}}w\neq0$.
    \end{enumerate}
    %\begin{enumerate}
        %\item 
        %$\bigbra{X_{31}^{\op{alg}}}^{\lambda_{j,1}-\lambda_{j,2}}\bigbra{X_{32}^{\op{alg}}}^{\lambda_{j,2}-\lambda_{j,3}}v_{\lambda_{j}}\neq0$ for $j\neq j_0,j_0-1$;
        %\item 
        %$\bigbra{X_{32}^{\op{alg}}}^{\max{0,\lambda_{j_0-1,1}-\lambda_{j_0-1,3}-(p-2)}}v_{\lambda_{j_0-1}}\neq0$;
        %\item 
        %$\bigbra{X_{31}^{\op{alg}}}^{\lambda_{j_0,1}-\lambda_{j_0,2}-1+\alpha_1}\bigbra{X_{32}^{\op{alg}}}^{\lambda_{j_0,2}-\lambda_{j_0,3}+\alpha_2}v_{\lambda_{j}}\neq0$.
    %\end{enumerate}
    Indeed, (i) follows from Lemma \ref{K-ext Lem nonzero alpha13}. For (ii), it follows from Lemma \ref{K-ext Lem nonzero alpha13} together with $\lambda_{j_0-1,1}-\lambda_{j_0-1,3}-(p-2)\leq\lambda_{j_0-1,2}-\lambda_{j_0-1,3}$ by assumption. For (iii), it follows from Lemma \ref{K-ext Lem nonzero alpha13} if $\alpha=\alpha_{13}$ and follows from Lemma \ref{K-ext Lem nonzero alpha12and23} if $\alpha=\alpha_{12}$. In the case $\alpha=0$, by \eqref{K-ext Eq I-structure on w31} we have $X_{23,j_0}w_{31,j_0}=-w_{21,j_0}$, which is nonzero by Lemma \ref{K-ext Lem I-structure}(i) since $k_{1,j_0}\neq k_{2,j_0}$. Then by Lemma \ref{K-ext Lem Y maps into} we have $X_{23}^{\op{alg}}w\neq0$, which implies (iii) by Lemma \ref{K-ext Lem nonzero alpha=0}.

    Therefore, we conclude that 
    \begin{equation*}
        0\neq\bbra{\sprod_{j\neq j_0-1}\bigbra{X_{31}^{(j),\op{alg}}}^{m_j}\bigbra{X_{32}^{(j),\op{alg}}}^{n_j}}\bigbra{X_{32}^{(j_0-1),\op{alg}}}^{n_{j_0-1}}w_{31,j_0}\in\ L(\lambda)_{w_0\mu'},
    \end{equation*}
    which completes the proof using \eqref{K-ext Eq Lem Xalpha injective 1}.
\end{proof}

Finally, we are ready to conclude. By Lemma \ref{K-ext Lem f(y) first weight} and \eqref{K-ext Eq Lem Xalpha injective 1}, we have 
\begin{equation*}
    f(y')=X_{13,j_0-1}f(y)\in\oplus_{\nu'}L(\lambda)_{\nu'}
\end{equation*}
with each $\nu'$ satisfies $\nu'_{j,2}\neq\mu_{j,2}=\mu'_{j,2}$ for some $j$. By Lemma \ref{K-ext Lem f(y) second weight}, we also know that $f(y')\in\sigma$ has nonzero image in $L(\lambda)_{\mu'}$. This is a contradiction and finishes the proof of Theorem \ref{K-ext Thm Calpha}.
\vspace{1em}

\noindent\textbf{Case 2.} $f=1$.

In this case, we provide the key construction and leave the details as an exercise. Since $\lambda\in C(\alpha,0)$ and $w_{31}\neq0$, by Lemma \ref{K-ext Lem I-structure} at least one of the following holds:
\begin{equation*}
\begin{split}
    \text{(i).}~&\alpha=\alpha_{13}~\text{and}~\lambda_1-\lambda_2\leq p-4;\\
    \text{(ii).}~&\alpha=\alpha_{12};\\
    \text{(iii).}~&\alpha=0~\text{and}~k_1\neq k_2;
\end{split}
\hspace{3em}
\begin{split}
    \text{(iv).}~&\alpha=\alpha_{13}~\text{and}~\lambda_2-\lambda_3\leq p-4;\\
    \text{(v).}~&\alpha=\alpha_{23};\\
    \text{(vi).}~&\alpha=0~\text{and}~k_2\neq k_3.
\end{split}
\end{equation*}
From now on, we concentrate on the cases (i),(ii),(iii), the other cases being symmetric. We consider the element
\begin{equation*}
    y\eqdef X_{13}^{p-2+\alpha_1}X_{12}^{\max\set{0,\lambda_1-\lambda_3-(p-3+\alpha_3)}}\smat{&&1\\&1&\\1&&}[1,v]\in\Ind_I^K W.
\end{equation*}
Arguing as in Lemma \ref{K-ext Lem image zero} and Lemma \ref{K-ext Lem f(y) first weight}, one can prove that $f(y)\in L(\lambda)_{\nu+(p-1)\alpha_{23}}$ with 
\begin{equation}\label{K-ext Eq f(y) first weight}
    \nu\eqdef w_0\bra{\lambda+\alpha}+(\alpha_1-1)\alpha_{13}+\max\set{0,\lambda_1-\lambda_3-(p-3+\alpha_3)}\alpha_{12}.
\end{equation}
Next, we consider the element (note that $X_{13}$ commutes with $X_{12}$)
\begin{equation*}
    y'\eqdef X_{13}^{2-\alpha_1}y=X_{12}^{\max\set{0,\lambda_1-\lambda_3-(p-3+\alpha_3)}}\smat{&&1\\&1&\\1&&}[1,w_{31}]\in\Ind_I^K W.
\end{equation*}
Arguing as in Lemma \ref{K-ext Lem f(y) second weight}, one can prove that $0\neq f(y')\in L(\lambda)_{\nu'}$ with
\begin{equation}\label{K-ext Eq f(y) second weight}
    \nu'\eqdef w_0(\lambda+\alpha)+\alpha_{13}+\max\set{0,\lambda_1-\lambda_3-(p-3+\alpha_3)}\alpha_{12}.
\end{equation}
By \eqref{K-ext Eq f(y) first weight} and Lemma \ref{K-ext Lem Y maps into}, we have 
\begin{equation*}
    f(y')=X_{13}^{2-\alpha_1}f(y)\in X_{13}^{2-\alpha_1}L(\lambda)_{\nu+(p-1)\alpha_{23}}\subseteq L(\lambda)_{\nu+(2-\alpha_1)\alpha_{13}+(p-1)\alpha_{23}}=L(\lambda)_{\nu'+(p-1)\alpha_{23}}.
\end{equation*}
By \eqref{K-ext Eq f(y) second weight}, we also have $f(y')\in L(\lambda)_{\nu'}\setminus\set{0}$. This is a contradiction and finishes the proof of Theorem \ref{K-ext Thm Calpha}.\qed

\section{The main result}\label{K-ext sec conclusion}

We keep the notation of the previous sections and finish the proof of Theorem \ref{K-ext Thm intro}. 

\begin{proposition}\label{K-ext Prop H1}
    There is an isomorphism of $K$-representations
    \begin{equation*}
        H^1(K_1/Z_1,\FF)\cong\oplus_{j=0}^{f-1}F(\alpha_{13})^{[j]}.
    \end{equation*}
\end{proposition}

\begin{proof}
    %By \cite[Lemma~5.1]{DdSMS99} we have 
    Since $K_1^p=K_2$, we have a $K$-equivariant isomorphism
    \begin{equation*}
    \begin{aligned}
        (K_1/Z_1)/(K_1/Z_1)^p&=(K_1/Z_1)/(K_2Z_1/Z_1)\cong\Mat_3(\Fq)/\Fq\\
        &\cong\bbbra{\Fq^{\oplus3}\otimes\bigbra{\Fq^{\oplus3}}^{*}}/\Fq\cong\bigbra{F(\alpha_{13})_{/\Fq}\oplus\Fq}/\Fq\cong F(\alpha_{13})_{/\Fq}.
    \end{aligned}
    \end{equation*}
    where the first isomorphism is given by $A\mapsto(A-\op{I}_3)\mod p$, and the third isomorphism follows from \cite[Lemma~2]{BDM15} and using $p\geq5$. Hence we have
    \begin{equation*}
    \begin{aligned}
        H^1(K_1/Z_1,\FF)&=\Hom(K_1/Z_1,\FF)=\Hom\bigbra{(K_1/Z_1)/(K_1/Z_1)^p,\FF}\\
        &\cong\Hom\bigbra{F(\alpha_{13})_{/\Fq},\FF}\cong F(\alpha_{13})_{/\Fq}^*\otimes_{\Fp}\FF\cong F(\alpha_{13})_{/\Fq}\otimes_{\Fp}\FF\cong\oplus_{j=0}^{f-1}F(\alpha_{13})^{[j]},
    \end{aligned}
    \end{equation*}
    which completes the proof.
\end{proof}

\begin{definition}
    Suppose that $f=1$. Let $\lambda,\lambda'\in X_1(T)$. We say that $(\lambda,\lambda')$ is a \textbf{bad pair} if one of the following holds: 
\begin{enumerate}
    \item 
    $\lambda'=\lambda+\alpha_{13}$ and $\lambda|_{\SL_3}\in\set{(p-2,p-2),(p-2,p-3),(p-3,p-2),(p-3,p-3)}$;
    \item
    $\lambda'=\lambda+\alpha_{12}$ and $\lambda|_{\SL_3}=(p-3,b)$ with $1\leq b\leq p-1$;
    \item 
    $\lambda'=\lambda+\alpha_{23}$ and $\lambda|_{\SL_3}=(a,p-3)$ with $1\leq a\leq p-1$;
    \item 
    $\lambda'=\lambda$ and $\lambda|_{\SL_3}=(a,b)$ with either $\set{a,b}\cap\set{p-2,p-1}\neq\emptyset$ or $a+b=p-2$.
    %\item 
    %$\set{(p-2,p-2),(p-1,p-1)}$;
    %\item 
    %$\set{(p-2,0),(p-1,1)}$ or $\set{(0,p-2),(1,p-1)}$;
    %\item 
    %$\set{(p-3,1),(p-1,0)}$ or $\set{(1,p-3),(0,p-1)}$;
    %\item 
    %$\set{(a,p-1),(a+2,p-2)}$ or $\set{(p-1,a),(p-2,a+2)}$ with $0\leq a\leq p-3$;
    %\item 
    %$\set{(a,b),(a,b)}$ with $\set{a,b,a+b}\cap\set{p-2,p-1}\neq\emptyset$.
\end{enumerate}
\end{definition}

\begin{definition}
    Suppose that $f\geq2$. Let $\lambda,\lambda'\in X_1(\un{T})$. We say that $(\lambda,\lambda')$ is a \textbf{bad pair} if one of the following holds: 
\begin{enumerate}
    \item 
    $\lambda'=\lambda+\alpha_{13,j_0}$ for some $j_0\in\cJ$ and $\lambda_{j_0-1}|_{\SL_3}=(p-1,p-1)$;
    \item
    $\lambda'=\lambda+\alpha_{12,j_0}$ for some $j_0\in\cJ$ and [either $\lambda_{j_0-1,1}-\lambda_{j_0-1,2}=p-1$, or $\lambda_{j_0}|_{\SL_3}=(p-3,1)$ and $\lambda_{j,1}-\lambda_{j,3}\geq p-1~\forall\,j\neq j_0$];
    \item
    $\lambda'=\lambda+\alpha_{23,j_0}$ for some $j_0\in\cJ$ and [either $\lambda_{j_0-1,2}-\lambda_{j_0-1,3}=p-1$, or $\lambda_{j_0}|_{\SL_3}=(1,p-3)$ and $\lambda_{j,1}-\lambda_{j,3}\geq p-1~\forall\,j\neq j_0$];
    \item 
    $\lambda'=\lambda$ and [either $\lambda_{j,1}-\lambda_{j,2}=p-1$ for some $j$, or $\lambda_{j,2}-\lambda_{j,3}=p-1$ for some $j$].
\end{enumerate}
\end{definition}

\begin{theorem}\label{K-ext Thm main}
    Let $\lambda,\lambda'\in X_1(\un{T})$ and $\sigma=F(\lambda)$, $\tau=F(\lambda')$ be two Serre weights for $\GL_3(\Fq)$. Suppose that both $(\lambda,\lambda')$ and $(\lambda',\lambda)$ are not bad pairs. Then we have
    \begin{equation*}
        \Ext^1_{K/Z_1}(\tau,\sigma)\cong\Ext^1_{\GL_3(\Fq)}(\tau,\sigma).
    \end{equation*}
\end{theorem}

\begin{proof}
    For all $K/Z_1$-representations $V$ we have $\Hom_{K/Z_1}(-,V)=\Hom_{\GL_3(\Fq)}(-,V^{K_1})$, hence the Grothendieck spectral sequence gives an exact sequence
    \begin{equation*}
        0\to\Ext^1_{\GL_3(\Fq)}(\tau,\sigma)\to\Ext^1_{K/Z_1}(\tau,\sigma)\to\Hom_{\GL_3(\Fq)}\!\bigbra{\tau,H^1(K_1/Z_1,\sigma)},
    \end{equation*}
    where the last group is isomorphic to $\oplus_{j=0}^{f-1}\Hom_{\GL_3(\Fq)}\!\bigbra{\tau,\sigma\otimes F(\alpha_{13})^{[j]}}$ by Proposition \ref{K-ext Prop H1}. In particular, the result follows if $\Hom_{\GL_3(\Fq)}\!\bigbra{\tau,\sigma\otimes F(\alpha_{13})^{[j]}}=0$ for all $j$. 

    From now on, we assume that $\Hom_{G(\Fp)}\!\bigbra{\tau,\sigma\otimes F(\alpha_{13})^{[j_0]}}\neq0$ for some $j_0\in\cJ$, i.e.\;$\tau\in\soc_{G(\Fp)}\bra{\sigma\otimes F(\alpha_{13})^{[j_0]}}$. By Corollary \ref{K-ext Cor Tensor Socle}, there exists $\alpha\in\Phi\cup\set{0}$ such that $\lambda'=\lambda+\alpha_{j_0}$ and $(\lambda_{j_0},\lambda'_{j_0})$ is a good pair in the sense of Definition \ref{K-ext Def good pair}.

    Suppose that $\alpha\in\Phi^+\cup\set{0}$. Then by Remark \ref{K-ext Rk Calpha+} we have $\lambda_{j_0}\in C(\alpha)^+$. Since $(\lambda,\lambda')$ is not a bad pair, a case-by-case examination shows that $\lambda\in C(\alpha,j_0)$ (see Definition \ref{K-ext Def Calpha} and Definition \ref{K-ext Def Calphaj0}). Hence we conclude by Theorem \ref{K-ext Thm Calpha}.

    Suppose that $\alpha\in\Phi^-$. Since we have 
    $\sigma^{\vee}\cong F(-w_0\lambda)$ and $\tau^{\vee}\cong F(-w_0\lambda')$, by duality it suffices to prove the result for the pair $(-w_0\lambda',-w_0\lambda)$. Then we have $-w_0\lambda=-w_0\lambda'+w_0\alpha_{j_0}$ with $w_0\alpha\in\Phi^+$. We also observe that $\lambda_{j_0}|_{\SL_3}=(a,b)$ implies $\bra{-w_0\lambda_{j_0}}|_{\SL_3}=(b,a)$. Since $(\lambda_{j_0},\lambda'_{j_0})$ is a good pair, by Remark \ref{K-ext Rk Calpha+} we have $\lambda'_{j_0}\in C(-\alpha)^+$. Then from the definition it is straightforward to check that $-w_0\lambda'_{j_0}\in C(w_0\alpha)^+$. Moreover, since $(\lambda',\lambda)$ is not a bad pair, from the definition we see that $(-w_0\lambda',-w_0\lambda)$ is also not a bad pair. Hence we are reduced to the case $\alpha\in\Phi^+$ discussed in the previous paragraph. This completes the proof.
\end{proof}

\bibliography{1}
\bibliographystyle{alpha}

\end{document}